\documentclass[11pt]{amsart}%{article}
\usepackage[english]{babel}
\usepackage[latin1]{inputenc}
\usepackage{float}
\usepackage[english, algosection, algoruled, noline]{algorithm2e}
\usepackage{latexsym}
\usepackage{amsmath}
\usepackage{amssymb}
\usepackage{graphicx}
\usepackage{amsthm}
\usepackage{amsfonts}
\usepackage[T1]{fontenc}

\usepackage{pdflscape}

\newtheorem{thm}{Theorem}[section]
\newtheorem{prop}[thm]{Proposition}
\newtheorem{cor}[thm]{Corollary}

\newtheorem{lem}[thm]{Lemma}
\newtheorem{defn}[thm]{Definition}
\newtheorem{example}[thm]{Example}
\newtheorem{remark}[thm]{Remark}

\newcommand{\R}{\mathbb{R}}

\newcommand{\C}{\mathbb{C}}
\newcommand{\N}{\mathbb{N}}

\newcommand{\K}{\mathbb{K}}
\newcommand{\Le}{\mathbb{Le}}
\newcommand{\XX}{\mathcal{M}}
\newcommand{\CC}{\mathcal{C}}
\newcommand{\Ac}{\mathcal{A}} 
\newcommand{\Lc}{\mathcal{L}} 
\newcommand{\Sc}{\mathcal{S}} 
\newcommand{\BB}{\mathcal{B}}
\newcommand{\QPK}{\K[\xx]/I}
\newcommand{\Mon}{\mathcal{M}}
\newcommand{\Span}[1]{\<{#1}\>}
\newcommand{\clK}{\overline{\mathbb K}}
\newcommand{\xx}{\mathbf x}

\newcommand{\SpanD}[2]{\<#1\,|\,{ #2}\>}

\newcommand{\piFB}{\pi_{F,\BB}}
\newcommand{\rank}{\mathrm{rank}}
\newcommand{\ev}{\ensuremath{\boldsymbol{1}}}
\newcommand{\tmmathbf}[1]{\ensuremath{\boldsymbol{#1}}}
\newcommand{\dual}[1]{{#1}^*}
\def\<{\langle}
\def\>{\rangle}

\begin{document}

%\Hacemos una portada
\title{Unconstraint global polynomial optimization via Gradient Ideal}

\author{Marta Abril Bucero, Bernard Mourrain, Philippe Trebuchet}
\address{Marta Abril Bucero, Bernard Mourrain: Galaad, Inria
  M\'editerran\'ee, 06902 Sophia Antipolis}
\address{Philippe Trebuchet: APR, LIP6, 4 place Jussieu 75005 Paris, France}

\maketitle

%\thispagestyle{empty}
%\newpage

%\Otro tipo de portada (las filas que tienen doble % es que van con % al ejecutarlo)
%\begin{titlepage}
%\begin{center}
%\vspace*{10 mm}
%{\Huge \textsc{SDFGHJKL}\vspace{5 mm}\\
%\textsc{DFGHJKs}\vspace{5 mm}\\
%\textsc{ Irreducibles}}\vspace{16 mm}\\
%\textit{Memoria presentada por}\vspace{3 mm}\\\textbf{SDFGHJKL}\vspace{3 mm}\\
%\textit{para obtener el t\'{\i}tulo de Doctor}\vspace{8 mm}\\
%\textit{Dirigida por}\vspace{3 mm}\\\textbf{UIOP}\vspace{40 mm}\\
%%\begin{figure}[h!]
%\begin{center}
%\includegraphics[width=3.5 cm]{reloj.jpg}
%\end{center}
%%\end{figure}
%{\large UNIVERSIDAD COMPLUTENSE DE MADRID}\vspace{2 mm}\\
%\textsc{Facultad de Ciencias Matem\'aticas, Departamento de \'Algebra}\\
%\textsc{Tesis Doctoral}\vspace{2 mm}\\
%Febrero 2010
%\end{center}
%\end{titlepage}

%\Comenzamos con un resumen
\begin{abstract}
  In this paper, we describe a new method to compute the minimum of a
  real polynomial function and the ideal defining the points which
  minimize this polynomial function, assuming that the minimizer ideal
  is zero-dimensional. Our method is a generalization of
  Lasserre relaxation method and stops in a finite number of steps.
  The proposed algorithm combines Border Basis, Moment Matrices and
  Semidefinite Programming.  In the case where the minimum is reached
  at a finite number of points, it provides a border basis of the
  minimizer ideal.
\end{abstract}

%\newpage
%\Introducimos un \'indice y escribimos los contenidos de cada secci\'on usando la orden \section
%\tableofcontents
%\thispagestyle{empty}
%\newpage
\setcounter{page}{1}

\section{Introduction}

Optimization appears in many areas of Scientific Computing, since the
solution of a problem can often be described as the minimum of an
optimization problem. 
Local methods such as gradient descent are often employed to handle 
global minimization problems. They can be very efficient to compute a local
minimum, but the output depends on the initial guess and 
they give no guarantee of a global solution.

In the case where the function $f$ to minimize is a polynomial,
it is possible to develop methods which ensure the 
computation of a global solution. Reformulating the problem as the
computation of a (minimal) critical value of the polynomial $f$, different polynomial
system solvers can be used to tackle it (see
e.g. \cite{Parrilo03minimizingpolynomial}, 
\cite{Greuet:2011:DRI:1993886.1993910}).
But in this case, the complex solutions of the underlying algebraic
system come into play and additional computation efforts should be
spent to remove these extraneous solutions.
Semi-algebraic techniques such as Cylindrical Algebraic
Decomposition or extensions \cite{ElDin:2008:CGO:1390768.1390781}
may also be considered here but suffer from similar issues.
Though the global minimization problem is known to be NP-hard (see
e.g. \cite{Nesterov2000}),
a practical challenge is to device methods which take into account ``only'' the real 
solutions of the problem or which can approximate them efficiently.

\noindent{}\textbf{Previous works.}
About a decade ago, a relaxation approach has been proposed in
\cite{Las01} (see also \cite{Par03}, \cite{Shor87}) to solve this
difficult problem. 
Instead of searching points where the polynomial $f$ reaches its
minimum $f^{*}$, a probability measure
which minimizes the function $f$ is searched. This problem is relaxed into a
hierarchy of finite dimensional convex minimization problems, that can
be solved by Semi-Definite Programming (SDP) techniques, and which
converges to the minimum $f^{*}$ \cite{Las01}. This hierarchy of
SDP problems can be formulated in terms of linear matrix inequalities
on moment matrices associated to the set of monomials of degree $\leq
t\in \N$ for increasing values of $t$.  The dual hierarchy can be
described as a sequence of maximization problems over the cone of
polynomials which are Sums of Squares (SoS). A feasibility condition is
needed to prove that this dual hierarchy of maximization problems also
converges to the minimum $f^{*}$, ie. that there is no duality gap.

From a computational point of view, this approach suffers from two drawbacks:
\begin{enumerate}
 \item the hierarchy of optimization problems may not be {\em exact}, ie. it
   may not always  convergence in a finite number of steps;
 \item the size of the SDP problems to be solved {\em  grows 
   exponentially} with $t$.
\end{enumerate}
To address the first issue, the following strategy has been considered:
add polynomial inequalities or equalities satisfied by the points where
the function $f$ is minimum. Inequality constraints can for instance be added 
to restrict the optimization problem to a compact subset of $\R^{n}$ and to make the
hierarchies exact \cite{Las01}, \cite{Marshall03}. 
Natural constraints which do not require apriori bounds on the
solutions are for instance the vanishing of the partial derivatives of $f$.
A result in \cite{Lau07} shows that if the gradient ideal (generated by
the  first differentials of $f$) is zero-dimensional, then 
the hierarchy extended with constraints from the gradient ideal is exact.
It is also proved that there is no duality gap if ``good'' generators of the gradient ideal
are used. 
%Conditions under which there is no duality gap in this type of
%optimization problems on semi-algebraic sets are studied in \cite{XXX}.
In \cite{NDS}, it is proved that the extended hierarchy is exact when
the gradient ideal is radical and in \cite{Nie11} the extended
hierarchy is proved to be exact for any polynomial $f$
when the minimum is reached in $\R^{n}$ (see also
\cite{Guo:2010:GOP:1837934.1837960}). In \cite{Schweighofer06}, the
relaxation techniques are analyzed for functions for which the minimum is
not reached and which satisfies some special properties ``at infinity''.

From an algorithmic point of view, this is not ending the investigations since
a criteria is needed to determine at which step the minimum is
reached. It is also important to known if all the minimizers can be
recovered at that step.

Methods based on moment matrices have been proposed to compute generators of
the ideal characterizing the real solutions of a polynomial system. In \cite{LLR08b}, 
the computation of generators of the (real) radical of an ideal is
based on moment matrices which involve all monomials of a given degree
and a stopping criterion related to Curto-Fialkow flat extension
condition \cite{CF96} is used.
This method is improved in \cite{lasserre:hal-00651759}.
Combining the border basis algorithm of \cite{Mourrain2005} and a weaker
flat extension condition \cite{MoLa2008}, a new algorithm which involves 
SDP problems of significantly smaller size is proposed to compute the
(real) radical of an ideal, when this (real) radical ideal is zero-dimensional.
In \cite{NDS}, an algorithm is also proposed to compute the global
minimum of the polynomial function $f$ based on techniques from 
\cite{Lau07}. It terminates when the gradient ideal is radical
zero-dimensional.

An interesting feature of these hierarchies of SDP problems is that,  at
any step they provide a lower bound of $f^{*}$ and the SoS hierarchy
gives certificates for these lower bounds (see e.g.
\cite{Kaltofen:2012:ECG:2069778.2070141} and reference therein). In
\cite{Hanzon01globalminimization, Jibetean:2005:SAG:1093657.1108704}
it is also shown how to
obtain ``good'' upper bounds by perturbation techniques, which can directly be
generalized to the approach we propose in this paper.

\noindent{}\textbf{Contributions.}
We show that if the minimum $f^{*}$ is reached in $\R^{n}$, a generalized
hierarchy of relaxation problems which involve the gradient ideal $I_{grad} (f)$
is exact and yields the generators of
the minimizer ideal $I_{min} (f)$ from a sufficiently high degree. 

In the case where the minimizer ideal is zero dimensional, we give a
criterion for deciding when the minimum is reached, based on the flat
extension condition in \cite{MoLa2008}. 

This criterion is used in a new algorithm which computes a border
basis of the minimizer ideal of a polynomial function, when this
minimizer ideal is zero-dimensional.

The algorithm is an extension of the real radical algorithm described
in \cite{lasserre:hal-00651759}.  The rows and
columns of the matrices involved in the semi-definite programming problem are
associated with the family of monomials candidates for being a
basis of the quotient space, i.e, a subset of
monomials of size much smaller than the number of monomials of the
same degree.
Thus, the size of the SDP problems involved in this computation is
significantly smaller than the one in \cite{Las01}, \cite{Lau07} or
\cite{NDS}.

We show that by solving this sequence of SDP problems, we obtain in a
finite number of steps the minimum $f^{*}$ of $f$, with no duality gap. When this
minimum is reached, the kernel of the Hankel matrix associated to the
solution of the SDP problem yields generators of the minimizer ideal which
are not in the gradient ideal.
Assuming that the minimizer ideal $I_{min} (f)$ is zero dimensional, 
computing the border basis of this kernel yields a representation
of the quotient algebra by $I_{min} (f)$ and thus a way to compute
effectively the minimizer points, using eigenvector solvers.

An implementation of this method has been developed, which integrates a border basis implementation 
and a numerical SDP solver. It is demonstrated on typical examples.

\noindent{}\textbf{Content.}
  The paper is organized as follows. Section 2 recalls the concepts of algebraic tools as ideals, varieties, dual space,
quotient algebra, the definitions and theorems about border basis,
and the Hankel Operators involved
in the computation of (real) radical ideals and in the computation of
our minimizer ideal. In Section 3, we describe the main results on
the exactness of the hierarchy of SDP problems on truncated Hankel
operators and show that 
the minimizer ideal can be computed from the kernel of these Hankel.
In Section 4, we analyze more precisely the case where the minimizer
ideal is zero-dimensional. In section 5, we describe our algorithm and
we prove its correctness. Finally in Section 6, we illustrate the
algorithm on some classical examples.

\section{Ideals, dual space, quotient algebra and border basis}\label{sec:2}
In this section, we set our notation and recall the eigenvalue techniques 
for solving polynomial equations and the border basis method. 

\subsection{Ideals and varieties}
Let $\K[\xx]$ be the set of the polynomials in the variables
$\xx=(x_1,\ldots$, $x_n)$, with coefficients in the field
$\K$. Hereafter, we will choose\footnote{For notational simplicity, we will consider only these two fields in this paper, but $\R$ and $\C$ can
be replaced respectively by any real closed field and any field containing its algebraic closure.}
 $\K=\R$ or $\C$.
 Let $\overline{\K}$ denotes the algebraic closure of ${\K}$.
For $\alpha \in \N^n$, $\xx^{\alpha}= x_1^{\alpha_1} \cdots x_n^{\alpha_n}$ is the monomial with exponent $\alpha$
and degree $|\alpha|=\sum_i\alpha_i$.  The set of all monomials in $\xx$ is
denoted $\Mon = \Mon(\xx)$. We say that $\xx^{\alpha} \le \xx^{\beta}$ if
$\xx^{\alpha}$ divides $\xx^{\beta}$, i.e., if $\alpha\le \beta$ coordinate-wise.
For a polynomial $f=\sum_\alpha f_\alpha \xx^\alpha$, its support is 
$supp(f):=\{\xx^\alpha\mid f_\alpha\ne 0\}$, the set of monomials occurring with a nonzero coefficient in $f$.

 For $t\in \N$ and $S\subseteq \K[\xx]$, we introduce the following sets:
\begin{itemize}
 \item $S_{t}$ is the set of elements of $S$ of degree $\le t$,
\item $\Span{S} = \big\{ \sum_{f\in S} \lambda_{f}\, f\ |\ f\in S, \lambda_f\in \K\big\}$ is the linear span of $S$,
\item $(S) = \big\{ \sum_{f\in S} p_f\, f \ | \ p_f \in \K[\xx], f \in S \big\}$ is the ideal in $\K[\xx]$ generated by $S$,
\item $\SpanD{S}{t} = \big\{ \sum_{f\in S_t} p_f\, f\ | \ p_f \in \K[\xx]_{t-\deg(f)}\big\}$ is the vector space spanned by $\{\xx^\alpha  f\mid
  f\in S_t, |\alpha|\le t-\deg(f)\}$, 
\item $S^+:=S\cup x_1S \cup \ldots \cup x_n S$ is the prolongation of $S$ by one degree,
\item $\partial S:= S^+\setminus S$ is the border of $S$,
\item $S^{[t]}:= S^{\stackrel{t\ \mathrm{times}}{+\cdots+}}$ is the result of applying $t$ times the prolongation operator `$^{+}$' on $S$, with 
$S^{[1]}=S^+$ and, by convention, $S^{[0]}=S$.
\end{itemize}
Therefore,
$S_t=S\cap \K[\xx]_t$, $S^{[t]}=\{x^\alpha f\mid f\in S, |\alpha|\le t\}$, 
$\SpanD{S}{t} \subseteq (S)\cap \K[\xx]_t= (S)_t$, but the inclusion may be strict.

If $\BB \subseteq \Mon$ contains $1$ then, for any monomial $m \in \Mon$, there exists an integer 
$k$ for which $m\in \BB^{[k]}$. The {\em $\BB$-index} of $m$, denoted by $\delta_{\BB}(m)$,
 is defined as the smallest integer $k$ for which  $m\in \BB^{[k]}$.

A set of monomials $\BB$ is said to be {\em connected to $1$} if $1\in \BB$ and for every monomial $m\neq 1$ in $\BB$, $m= x_{i_{0}} m'$ for
some $i_{0}\in [1,n]$ and $m'\in \BB$. 

Given a vector space $E \subseteq \K$, its prolongation $ E^+ : = E + x_1 E + \ldots + x_n E$ is again a vector space.

The vector space $E$ is said to be \textit{connected to 1} if $1\in E$ and any non-constant polynomial $p \in E$ can be written as $p = p_0+ \sum_{i = 1}^n x_i p_i$
for some polynomials $p_0,p_i\in E$ with $\deg(p_0)\le \deg(p)$, $\deg(p_i)\le \deg(p)-1$ for $i\in [1,n]$.
 
Obviously, $E$ is connected to 1 when $E=\Span{\CC}$ for some monomial set $\CC\subseteq \Mon$ which is connected to 1.  
Moreover, $E^+ = \Span{\CC^+}$ if $E=\Span{\CC}$.

\medskip
Given an ideal $I\subseteq \K[\xx]$ and a field $\Le \supseteq \K$, we denote by
\begin{equation*}
 V_{\Le}(I):=\{x\in \Le^n\mid f(x)=0\ \forall f\in I\}
\end{equation*}
its associated variety in $\Le^{n}$. By convention $V(I)=V_{\clK}(I)$. For a set
$V\subseteq \K^n$, we define its vanishing ideal
\[I(V):=\{f\in \K[\xx]\mid f(v)=0\ \forall v\in V\}.\]
Furthermore, we denote by
\[\sqrt I:=\{f\in \K[\xx] \,\mid\, f^m\in I \ \text{ for some } m\in \N\setminus \{0\}\}\]
the radical of $I$. 

For $\K=\R$, we have $V(I)=V_{\C}(I)$, but one may also be interested in the subset of real solutions, namely the real variety
$V_\R(I)=V(I)\cap \R^n.$
The corresponding vanishing ideal is $I(V_\R(I))$ and the \emph{real radical ideal} is
\[\sqrt[\R]I:=\{p\in \R[\xx]\mid p^{2m} +\sum_j q_j^2\in I \ \text{ for some }q_j\in \R[\xx], m\in \N\setminus \{0\}\}.\]
Obviously,
\[I\subseteq \sqrt I\subseteq I(V_{\C}(I)),\ \  I\subseteq \sqrt[\R]I\subseteq I(V_\R(I)).\]
An ideal $I$ is said to be {\em radical} (resp., {\em real radical}) if $I=\sqrt I$ (resp. $I= \sqrt[\R]I$). Obviously, $I\subseteq
I(V(I))\subseteq I(V_\R(I))$. Hence, if $I\subseteq \R$ is real radical, then $I$ is radical and moreover, $V(I)=V_\R(I)\subseteq
\R^n$ if $|V_\R(I)|<\infty$.

The following two famous theorems relate vanishing and radical
ideals:
\begin{thm}\ 
\begin{itemize}
\item[(i)] {\bf Hilbert's Nullstellensatz} (see, e.g., \cite[\S 4.1]{CLO97})
  $\sqrt I=I(V_{\C}(I))$ for an ideal $I\subseteq \C[\xx]$.
\item[(ii)] {\bf Real Nullstellensatz} (see, e.g., \cite[\S 4.1]{BCR98})
$ \sqrt[\R]I=I(V_\R(I))$ for  an  ideal $I\subseteq \R[\xx]$.
\end{itemize}
\end{thm}

\subsection{The roots from the quotient algebra structure}
Given an ideal $I\subseteq \K[\xx]$, the quotient set $\QPK$ consists of all
cosets $[f]:=f+I=\{f+q \mid q \,\in\,I\}$ for $f\in \K[\xx]$, i.e., all equivalent
classes of polynomials of $\K[\xx]$ modulo the ideal $I$.  The quotient set
$\QPK$ is an algebra with addition $[f]+[g]:=[f+g]$, scalar multiplication
$\lambda [f]:=[\lambda f]$ and with multiplication $[f][g]:=[fg]$, for
$\lambda\in \R$, $f,g\in \K[\xx]$.

We will say that an ideal $I$ is zero-dimensional
if $0<|V_{\clK}(I)|<\infty$. Then
$\QPK$ is a finite-dimensional vector space and its dimension is the
number of roots counted with multiciplicity (see e.g. \cite{CLO97}, \cite{em-07-irsea}).
Thus $|V_{\clK}(I)|\le \dim_{\K}\QPK$, with equality if and only if $I$ is radical.

Assume that $0< |V_{\clK}(I)|<\infty$ and set $N:=\dim_{\K}\QPK$ so that
$|V_{\clK}(I)|\le N<\infty$. Consider a set  
 $\BB:=\{b_1,\ldots,b_N\}\subseteq \K[\xx]$ for which 
$\{[b_1],\ldots,[b_N]\}$ is a basis of $\QPK$; by abuse of language we also say that $\BB$ itself is a basis of
$\QPK$. Then every $f\in \K[\xx]$ can be written in a unique way as $f=\sum_{i=1}^N c_i b_i +p,$ where $c_i\in \K,$ $ p\in I;$
the polynomial $\pi_{I,\BB}(f):=\sum_{i=1}^N c_i b_i$ is called the remainder of $f$ modulo
$I$, or its {\em normal form}, with respect to the basis $\BB$. In other words, $\Span{\BB}$ and $\QPK$ are isomorphic vector spaces. 

Given a polynomial $h\in \K[\xx]$, we can define the {\em multiplication (by $h$) operator} as
\begin{equation}
\label{mult}
\begin{array}{lccl}
\XX_h: & \QPK & \longrightarrow & \QPK\\
     &  [f] & \longmapsto &\XX_h([f])\,:=\, [hf]\,,
\end{array}
\end{equation}
Assume that $N:=\dim_{\K} \QPK<\infty$. Then the multiplication operator $\XX_h$ can be represented by its matrix, again denoted $\XX_h$ for simplicity, with respect to a given basis 
$\BB=\{b_1,\ldots,b_N\}$ of $\QPK$. 

Namely, setting $\pi_{I,\BB}(hb_j):= \sum_{i=1}^N a_{ij}b_i$ for some scalars $a_{ij}\in \K$,
 the $j$th column of $\XX_h$ is the vector $(a_{ij})_{i=1}^N$.

Define the vector
$\zeta_{\BB,v}:=(b_j(v))_{j=1}^N \in \clK^N$, whose coordinates are the
evaluations of the polynomials $b_j \in \BB$ at the point $v\in \clK^n$. The
following famous result (see e.g., \cite[Chapter 2\S4]{CLO98}, \cite{em-07-irsea}) relates the eigenvalues of the multiplication operators in $\QPK$ to the algebraic variety $V(I)$.
This result underlies the so-called eigenvalue method for solving
polynomial equations \cite{SGPhT09} and plays a central role in many
algorithms, also in the present paper. 

\begin{thm}
\label{thm::muloperator} Let $I$ be a zero-dimensional ideal in $\K[\xx]$, $\BB$ a basis of $\QPK$, and $h\in \K[\xx]$.
The eigenvalues of the multiplication operator $\XX_h$ 
are the evaluations $h(v)$ of the polynomial $h$ at the points $v \in V(I)$.
Moreover, $(\XX_h)^T\zeta_{\BB,v}=h(v) \zeta_{\BB,v}$ and
the set of common eigenvectors of $(\XX_h)_{h\in \K[\xx]}$ are up to a
non-zero scalar multiple the
vectors $\zeta_{\BB,v}$ for $v\in V(I)$.
\end{thm}

Throughout the paper we also denote by $\XX_i:=\XX_{x_i}$ the matrix of the multiplication operator by the variable $x_i$. 
By the above theorem, the eigenvalues of the matrices $\XX_{i}$ are the $i$th coordinates of the points $v\in V(I)$. 
Thus the task of solving a system of polynomial equations is reduced to a task of numerical linear algebra once a basis of $\QPK$ and a normal form algorithm are available, permitting the construction of the multiplication matrices $\XX_i$.

\subsection{Border bases}
The eigenvalue method for solving polynomial equations from the above section  requires
the knowledge of a basis of $\QPK$ and an algorithm to compute the normal form of
a polynomial with respect to this basis. In this section we will recall a
general method for computing such a basis and a method to reduce polynomials
to their normal form. 

Throughout $\BB\subseteq \Mon$ is a finite set of monomials.

\begin{defn}
A rewriting family $F$ for a (monomial) set $\BB$ is a set of polynomials  $F=\{f_i\}_{i\in \mathcal{I}}$ such that 
\begin{itemize}
\item $supp(f_i)\subseteq \BB^+$,
\item $f_i$ has exactly {\bf one} monomial in $\partial \BB$, denoted as $\gamma(f_i)$  and called the {\it leading monomial} of $f_i$. (The polynomial $f_i$ is
normalized so that the coefficient of $\gamma(f_i)$ is $1$.)
\item  if  $\gamma(f_i)=\gamma(f_j)$ then $i=j$.
\end{itemize}
\end{defn}

\begin{defn}
We say that the rewriting family $F$ is \emph{graded} if $\deg (\gamma(f)) =\deg
(f)$ for all $f\in F$.
\end{defn}
 
\begin{defn}\label{defnormfam}
A rewriting family $F$ for $\BB$ is said to be \emph{complete} in degree $t$ if
it is  graded and satisfies  $(\partial \BB)_t \subseteq \gamma(F)$; that is,
each monomial $ m \in \partial \BB$ of  degree at most $t$ is the leading monomial of some (necessarily unique)
 $f\in F$.
\end{defn}

\begin{defn} \label{defpiBF}
Let $F$  be  a rewriting family for $\BB$, {complete} in degree $t$.
Let $\pi_{F,\BB}$ be the projection on $\Span{\BB}$ along $F$ defined recursively on the
monomials $m \in \Mon_{t}$ in the following way:
\begin{itemize}
  \item if $m \in \BB_{t}$, then $\pi_{F,\BB}(m)=m$,
 \item if $m \in (\partial \BB)_{t} \ (= (\BB^{[1]}\setminus \BB^{[0]})_{t})$,
then  $\pi_{F,\BB}(m)=m-f$, where $f$ is the (unique) polynomial in $F$ for which
 $\gamma(f)=m$,
 \item if $m \in (\BB^{[k]}\setminus \BB^{[k-1]})_{t}$ for some integer $k\ge 2$, write $m = x_{i_{0}} m'$, where 
$m' \in \BB^{[k-1]}$ and $i_{0}\in [1,n]$ is the smallest possible variable index for which such a decomposition exists,
then 
   $\pi_{F,\BB}(m)=\pi_{F,\BB}(x_{i_{0}}\,\pi_{F,\BB}(m') )$. 
\end{itemize}
\end{defn}
If $F$ is a graded rewriting family, one can easily verify that $\deg(\piFB(m))\le \deg(m)$ for $m\in\Mon_t$. The map $\piFB$ extends by linearity to a linear map from 
$\K[\xx]_t$ onto $\Span{\BB}_t$. By construction, 
$f=\gamma(f)- \piFB(\gamma(f))$ and $\piFB(f)=0$ for all $f\in F_t$. 
The next theorems show that, under some natural commutativity condition,
the map $\piFB$ coincides with the linear projection from $\K[\xx]_t$ onto $\Span{\BB}_{t}$
along the vector space $\SpanD{F}{t}$. It leads to the notion of border basis.

\begin{defn} \label{def:borderbasis}
Let $\BB\subset \Mon$ be connected to $1$. 
A family $F\subset \K[\xx]$ is a border basis for $\BB$ if it is  a rewriting family
for $\BB$, {complete} in all degrees, and such that $\K[\xx]= \Span{\BB} \oplus (F).$
\end{defn}

An algorithmic way to check that we have a border basis is based on
the following result, that we recall from \cite{Mourrain2005}:
\begin{thm}\label{thmcom}
Assume that $\BB$ is connected to $1$ and let $F$ be a rewriting family for
 $\BB$, complete in degree $t\in \N$. Suppose that, for all $m\in \Mon_{t-2}$,
\begin{equation}\label{eqcom}
   \pi_{F,\BB}(x_i\,\pi_{F,\BB}(x_j\,m))=\pi_{F,\BB}(x_j\,\pi_{F,\BB}(x_i\,m))
\ \text{ for all } i,j\in[1,n].
\end{equation}
Then $\pi_{F,\BB}$ coincides with the linear projection of $\K[\xx]_{t}$ on $\Span{\BB}_{t}$
along the vector space $\SpanD{F}{t}$; that is,
$\K[\xx]_t=\Span{\BB}_t\oplus \SpanD{F}{t}.$
\end{thm}

In order to have a simple test and effective way to test the
commutation relations \eqref{eqcom},
we introduce now the commutation polynomials.
\begin{defn}
Let $F$ be a rewriting family and $f,f'\in F$.
Let $m,m'$ be the smallest degree monomials for which 
$m\,\gamma(f)=m'\, \gamma(f')$.
Then the polynomial 
$C(f,f'):= m f-m'f' = m' \piFB(f')-m\piFB(f)$ is called the \emph{commutation  polynomial} of $f,f'$.
\end{defn}
 
\begin{defn}
For a rewriting family $F$ with respecet to $\BB$, we denote by $C^{+}(F)$ the set of polynomials
of the form $m\, f - m'\, f',$ where $f,f'\in F$ and 
\\ $m,m' \in \{0,1,x_{1},\ldots,x_{n}\}$ satisfy 
\begin{itemize}
 \item either $m\, \gamma(f)=m'\, \gamma(f')$,
 \item or  $m\, \gamma(f)\in \BB$ and $m'=0$.
\end{itemize}
\end{defn}
Therefore, $C^{+}(F) \subset \Span{\BB^{+}}$ and $C^+(F)$ contains all commutation
polynomials $C(f,f')$ for $f,f'\in F$ whose monomial multipliers $m, m'$ are
of degree $\leq 1$.
The next result can be deduced using Theorem \ref{thmcom}.

\begin{thm}\label{thmnfdegd}
Let $\BB\subset \Mon$ be connected to $1$ and let $F$ be a rewriting family for $\BB$, complete in degree $t$. 
If for all $c\in C^{+}(F)$ of degree $\le t$, $\pi_{F,\BB}(c) = 0$, 
then $\pi_{F,\BB}$  is the projection of $\K_{t}$  on $\Span{\BB}_{t}$ along
$\SpanD{F}{t}$, ie. $\K_{t}=\Span{\BB}_{t}\oplus \SpanD{F}{t}$. 
\end{thm}
 
If such a property is satisfied we say that $F$ is a border basis for
$B$ in degree $\leq t$.

\begin{thm}[border basis, \cite{Mourrain2005}]\label{thmnfanyt} 
Let $\BB\subset \Mon$ be connected to 1 and let $F$ be a rewriting family for $\BB$, complete in any degree. Assume that $\pi_{F,\BB}(c) = 0$ for all 
$c \in C^{+}(F)$.Then $\BB$ is a basis of $\K[\xx]/(F)$, $\K=\Span{\BB} \oplus (F)$, and $(F)_{t}=\SpanD{F}{t}$ for all $t\in \N$; the set $F$ is 
a \emph{border basis} of the ideal $I=(F)$ with respect to $\BB$.
\end{thm}

This implies the following characterization of border bases using 
the commutation property.
\begin{cor}[border basis, \cite{m-99-nf}]\label{thmcomanyt} 
Let $\BB\subset \Mon$ be connected to 1 and let $F$ be a rewriting family for $\BB$, complete in any degree.
If for all $m\in \BB$ and all indices $i,j\in[1,n]$, we have:
\begin{equation*}
   \pi_{F,\BB}(x_i\,\pi_{F,\BB}(x_j\,m))=\pi_{F,\BB}(x_j\,\pi_{F,\BB}(x_i\,m)),
\end{equation*}
then $\BB$ is a basis of $\K/(F)$, $\K=\Span{\BB} \oplus (F)$, and
$(F)_{t}=\SpanD{F}{t}$ for all $t  \in \N$.
\end{cor}

\subsection{Hankel operators and positive linear forms}
This section is based on \cite{lasserre:hal-00651759}.
\begin{defn}
 For $\Lambda \in \R[\xx]^*$, the Hankel operator $H_{\Lambda}$ is the operator from $\R[\xx]$ to $\R[\xx]^*$ defined by
 \begin{equation}
  H_{\Lambda} : p \in \R[\xx] \mapsto p\cdot \Lambda \in \R[\xx]^*
 \end{equation}
\end{defn}

\begin{defn}
 We define the kernel of the Hankel operator:
 \begin{equation}
  \ker H_{\Lambda}=\{p \in \R[\xx] \mid p \cdot \Lambda =0, i.e, \ \Lambda(pq)=0 \ \forall q \in \R[\xx] \}
 \end{equation}
\end{defn}

% \begin{proof}
% Direct verification, using $H_\Lambda$ as isomorphism in the proof of the second part of the lemma.
% \end{proof}
To analyse the properties of $\Lambda$, we study the quotient algebra
$\K[\xx]/\ker H_\Lambda=\Ac_{\Lambda}$. A first result is the following
(see e.g. \cite{lasserre:hal-00651759}):
\begin{lem}\label{lem:basis}
The rank of the operator $H_{\Lambda}$ is finite if and only if $\ker H_\Lambda$ is a zero-dimensional ideal, in which case $\dim \K[\xx]/\ker H_\Lambda=rank H_{\Lambda}$.
\end{lem}

% \begin{proof}
% Directly from the fact that, given $p_1,\ldots,p_r\in\K[\xx]$, $H_\Lambda(p_1),\ldots,H_\Lambda(p_r)$ are linearly independent in 
% $\K[\xx]^*$ if and only if the cosets $[p_1],\ldots,[p_r]$ are linearly independent in\\
% $\K[\xx]/\ker H_\Lambda$.
% \end{proof}

For a zero-dimensional ideal $I\subset \K[\xx]$ with simple zeros $V(I)= \{\zeta_1,\ldots,\zeta_r\} \subset \K^n$, we have
$I^{\bot} = \Span{ \tmmathbf{1}_{\zeta_1},\ldots,\tmmathbf{1}_{\zeta_r}}$ and the ideal $I$ is radical as a consequence of Hilbert's Nullstellensatz.
When $I=\ker H_{\Lambda}$, this yields the following property. 
\begin{prop}\label{propradideal}
Let $\K=\C$ and assume that $rank H_{\Lambda} = r < \infty$. Then, the ideal $\ker H_\Lambda$ is radical if and only if
\begin{equation}\label{eqlambdarad}
\Lambda=\sum_{i=1}^r \lambda_i\tmmathbf{1}_{ \zeta_i} \ \text{ with } \lambda_i\in \K-\{0\} \ \text{ and } \zeta_i\in \K^n \text{ pairwise distinct},
\end{equation}
in which case $\ker H_\Lambda=I(\zeta_1,\ldots,\zeta_r)$ is the
vanishing ideal of the points $\zeta_1, \ldots, \zeta_{r}$.
\end{prop}

As a corollary, we deduce that when $\K = \R$ and $rank H_\Lambda
=r<\infty$, the ideal $\ker H_\Lambda$ is real radical if and only if
the points $\zeta_i$ are in $\R^n$.

% \begin{proof}
% If $\ker H_\Lambda$ is real radical then $V(\ker H_\Lambda)=\{\zeta_1,\ldots,\zeta_r\}\subset \R^n$, so that \eqref{eqlambdarad} gives \eqref{eqlambdarealrad}. 
% Conversely, if $\Lambda$ is as in (\ref{eqlambdarealrad}), then $\ker H_\Lambda$ is real radical, since $\sum_j q_j^2 \in \ker H_\Lambda$ implies $\sum_jq_j(\zeta_i)^2=0$ and thus
% $q_j(\zeta_i)=0$, giving $q_j\in \ker H_\Lambda$.
% \end{proof}

\begin{defn}
  We say that $\Lambda \in \R[\xx]^*$ is positive, which we denote $\Lambda \succcurlyeq 0$, if $\Lambda (p^2) \ge 0$ for all
  $p\in \R[\xx]$. We have that $\Lambda\succcurlyeq 0$ iff
  $H_{\Lambda} \succcurlyeq 0$. If moreover $\Lambda (1)=1$, we say
  that $\Lambda$ is a probability measure.
\end{defn}
This term is justified by a theorem of Riesz-Haviland \cite{rudin-functional}, which
states that a Lebesgue measure on $\R^{n}$ is uniquely determined by its value
on the polynomials $\in \R[\xx]$.
In particular, if $\Lambda\succcurlyeq 0$ and
$\Lambda (1)=1$, there exists a unique probability measure $\mu$ such
that $\forall p\in \R[\xx], \Lambda (p) = \int p\, d\mu$.

An important property of positive forms is the following:
% \begin{prop} \label{prop:radreal}
% If  $\Lambda \succcurlyeq 0$, then $\ker H_{\Lambda}$ is a real radical ideal.
% \end{prop}

% \begin{proof}
% Assume $\sum_i p_i^2\in \ker H_\Lambda$; we show that $p_i\in\ker H_\Lambda$. Indeed, $(\sum_ip_i^2)\cdot \Lambda=0$
% implies, for all $q\in \R[\xx]$, $0=\Lambda(\sum_ip_i^2q^2)=\sum_i\Lambda(p_i^2q^2)$ and thus $\Lambda(p_i^2q^2)=0$. By Lemma \ref{positive}, this in turn implies $\Lambda(p_iq)=0$ and thus $p_i\in
% \ker H_\Lambda$.
% \end{proof}

\begin{prop}\label{proppositive}
Assume $rank H_{\Lambda} = r < \infty$. Then $\Lambda \succcurlyeq 0$ if and only if $\Lambda$ has a decomposition
(\ref{eqlambdarad}) with $\lambda_i >0$ and distinct $\zeta_i\in
\R^n$, in which case $V(\ker H_{\Lambda}) =\{\zeta_1, \ldots,
\zeta_r\}\subset \R^n$.
\end{prop}
In particular, it shows that if  $\Lambda \succcurlyeq 0$, then $\ker H_{\Lambda}$ is a real radical ideal.

% \begin{proof}
% If $\Lambda=\sum_{i=1}^r\lambda_i \tmmathbf{1}_{\zeta_i}$ with $\lambda_i>0$ and $\zeta_i\in\R^n$, then $\Lambda \succcurlyeq 0$
% holds obviously. Conversely, assume that $\Lambda \succcurlyeq 0$ then by Proposition \ref{prop:radreal} the ideal $\ker H_\Lambda$ is
% real radical, so that it is radical and $V (\ker H_{\Lambda})\subset
% \R^{n}$. By Proposition \ref{propradideal}, $\Lambda$ has a
% decomposition (\ref{eqlambdarad}) where $\zeta_i\in\R^n$.
% Let $p_i$ be a family of interpolation polynomials at the points $\zeta_i$.

% As $\lambda_i=\Lambda(p_i^{2})\ge 0$, and $p_i$ are
% interpolation polynomials at the $\zeta_i$'s. As $p_i^2-p_i\in I(\zeta_1,\ldots,\zeta_r)=\ker H_\Lambda$, we have
% $\lambda_{i}=\Lambda(p_i)=\Lambda(p_i^2)\ge 0$, which concludes the proof.
% \end{proof}

\section{Main Results}
In this section, we give the main results which shows how the minimum
of $f$ and the ideal defining the points where this minimum is reached,
can be computed. 

Hereafter, we will assume that the minimum $f^{*}$ of $f$ is reached at a
point $\xx^{*}\in \R^{n}$.
\begin{defn}
 We define the gradient ideal of $f(x)$:
 \begin{equation}
  I_{grad}(f)=(\nabla f(x))=(\frac{\partial f}{\partial x_1},..,\frac{\partial f}{\partial x_n}).
 \end{equation}
\end{defn}

\begin{defn}
 We define the minimizer ideal of f(x):
 \begin{equation}
  I_{min}(f)= I(\xx^* \in \R^{n} \ s.t. \ f(\xx^*) \ is \ minimum).
 \end{equation}
\end{defn}
By construction, $I_{grad} (f)\subset I_{min} (f)$ and $I_{min}
(f)\neq (1)$ if the minimum $f^{*}$ is reached in $\R^{n}$.
The objective of this section is to describe a method to compute
generators of $I_{min} (f)$ from generators of $I_{grad} (f)$.
For that purpose, first of all we need 
to restrict our analysis to matrices of finite size. For this reason,
we consider here truncated Hankel operators, which will play a central role in the construction of the minimizer ideal of $f$.
\begin{defn}
 For a vector space $E\subset \R[\xx]$, let $E\cdot E:=\{p\,q\mid p,q\in E\}$. 
For a linear form $\Lambda \in \Span{E \cdot E}^{*}$, we define the map $H_{\Lambda}^E : E \rightarrow
 E^{\ast}$ by $H_{\Lambda}^E(p) = p \cdot \Lambda$ for $p\in E$.  Thus
 $H_{\Lambda}^E$ is called a truncated Hankel operator, defined on the subspace $E$.
\end{defn}

\begin{defn}
 We define the kernel of the truncated Hankel operator:
 \begin{equation}
  \ker H_{\Lambda}^E=\{p \in E \mid p \cdot \Lambda =0, i.e, \ \Lambda(pq)=0 \ \forall q \in E \}.
 \end{equation}
\end{defn}
When $E=\R[\xx]_t$ for $t\in \N$, $H_{\Lambda}^E$ is also denoted $H_{\Lambda}^t$.

Given a subspace $E_0\subset E$, $\Lambda$ induces a linear map on $\Span{E_0\cdot E_0}$ and we can consider
the induced truncated Hankel operator $H^{E_0}_\Lambda: E_0
\rightarrow (E_0)^*$ as a restriction of $H^{E}_{\Lambda}$.

Below we will deal with the following sets and minimums, in order to define our primal-dual problems. 

\begin{defn}
 Given a vector space $E \subset \R[\xx]$ and $G \subset \langle E
 \cdot E \rangle$, we define 
 \begin{equation}
  \Lc_{G,E,\succcurlyeq}:=\{\Lambda \in \langle E \cdot E \rangle^* \mid \Lambda \perp G, \ \Lambda(1)=1, \ \Lambda(p^2) \ge 0 \ \forall p \in E \}. \label{eq1}
 \end{equation}
 If $E=\R[\xx]_t$ and $G'=\SpanD{G}{2t}$, we also denote $\mathcal{L}_{G',E,\succcurlyeq}$ by $\mathcal{L}_{G,t,\succcurlyeq}$.
\end{defn}
Notice that if $G \subset G'$ and $E \subset E'$ then
$\mathcal{L}_{G',E',\succcurlyeq} \subset
\mathcal{L}_{G,E,\succcurlyeq}$.   When $E$ and $G$ are vector spaces
of finite dimension, $\Lc_{G,E, \succcurlyeq}$ is the intersection of the closed
convex cone of semi-definite positive quadratic forms on $E\times E$ with a linear space,
thus it is a convex closed semi-algebraic set. More details on its
description will be given in Section \ref{sec-0dim}.
We will need the following result:

\begin{lem}\label{positive}
For any vector space $E$ and $G=\{0\}$ and $\Lambda,\Lambda' \in
\Lc_{G,E,\succcurlyeq}$,  we have:
\begin{itemize}
 \item $\forall p\in E$, $\Lambda(p^2) = 0$ implies $p \in \ker H_{\Lambda}^{E}$.
 \item $\ker H_{\Lambda+\Lambda'}^{E}=\ker H_\Lambda^{E} \cap \ker H_{\Lambda'}^{E}$. 
\end{itemize}
\end{lem}

\begin{proof} The first point is a consequence 
of the positivity of $H_\Lambda^{E}$.

For the second point, $\forall p, q\in E$, $\forall t\in\R$, $\Lambda((p+tq)^2) = t^2 \Lambda(q^2) +
2t \Lambda(p\,q) \ge 0$. Dividing by $t$ and letting $t$ go to zero
yields $\Lambda (p\,q)=0$, thus showing $p\in\ker H_\Lambda^{E}$.
The inclusion $\ker H_\Lambda^{E} \cap \ker H_{\Lambda'}^{E} \subset \ker H_{\Lambda+\Lambda'}^{E}$ is obvious.

Conversely, let $p\in \ker H_{\Lambda+\Lambda'}$. In particular,
$(\Lambda+\Lambda')(p^2)=0$, which implies $\Lambda(p^2)=\Lambda'(p^2)=0$ (since 
$\Lambda(p^2),\Lambda'(p^2)\ge 0$) and thus $p\in \ker H_\Lambda \cap \ker H_{\Lambda'}$.
\end{proof}

\begin{defn}
 Given a vector space $E \subset \R[\xx]$ and $G \subset \langle E
 \cdot E \rangle$, we define 
 \begin{equation}
  \Sc_{G,E}:=\{\ p \in \R[\xx] \mid p= \sum_{i=1}^s h_i^2 + h, \ h_i \in E, \ h \in G  \} \label{eq2}.
 \end{equation}
 If $E=\R[\xx]_t$ and $G'=\SpanD{G}{2t}$, we also denote $\Sc_{G',E}$ by $\Sc_{G,t}$.
\end{defn}
Notice that if $G \subset G'$ and $E \subset E'$ then $\Sc_{G,E} \subset \Sc_{G',E'}$.
When $E$ and $G$ are vector spaces of finite dimension, $\Sc_{G,E}$
 is the projection of the sum of a linear space and the 
convex cone of positive quadratic forms on $\dual{E}\times \dual{E}$.

\begin{defn}\label{mintruncated} %\label{DefG}
 Let $E$ be a subspace of $\R[\xx]$ such that $ 1 \in E$ and $f \in \langle E \cdot E \rangle$ and let $G \subset \langle E \cdot E \rangle$.
 We assume that $f$ attains its minimum at some points $\xx^* \in \R^n$. We define the following extrema:
  \begin{itemize}
  \item $f^*= \min_{x\in\R^n} f(x),$
  \item $f^{\mu}_{G,E}= \inf \ \{ \Lambda(f)$\footnote{the notation
      $^\mu$ stands for measure} s.t.  $\Lambda \in\Lc_{G ,E,\succcurlyeq} \},$
  \item $f^{sos}_{G,E}= \sup \ \{ \lambda\in \R$ s.t. $f-\lambda \in  \Sc_{G,E}\}.$  
 \end{itemize}
 If $E=\R[\xx]_t$ (resp. $E=\R[\xx]$) and $G'=\SpanD{G}{2t}$ we also denote $f^{\mu}_{G',E}$ by $f^{\mu}_{G,t}$ (resp. $f^{\mu}_{G}$) and $f^{sos}_{G',E}$ by $f^{sos}_{G,t}$ (resp. $f^{sos}_{G}).$ 
\end{defn}

\begin{remark}
 By convention if the sets are empty, $f^{sos}_{G,E} = -\infty$ and $f^{\mu}_{G,E}= +\infty$.
\end{remark}

We now analyse the relations between these different extrema.

\begin{remark}\label{subset}
 Let $E \subset E'$ be two subspaces of  $\R[\xx]$ and $G
 \subset \langle E \cdot E \rangle$, $G' \subset \langle E' \cdot E'
 \rangle$ with $G \subset G'$ then we directly deduce that 
 \begin{itemize}
  \item $f^{\mu}_{G,E} \le f^{\mu}_{G',E'}$,
  \item $f^{sos}_{G,E} \le f^{sos}_{G',E'}$, 
 \end{itemize}
from the fact that $\mathcal{L}_{G',E',\succcurlyeq} \subset \mathcal{L}_{G,E,\succcurlyeq}$ and $\mathcal{S}_{G,E} \subset \mathcal{S}_{G',E'}$.
\end{remark}

If we take $G\subset {{I}}_{min}(f)$, we have the following relations between the extrema.

\begin{prop}\label{minor}
 Let $G \subset {I}_{min}(f)$. Then $f^{sos}_{G,E} \le f^{\mu}_{G,E} \le f^*$.  
\end{prop}

\begin{proof}\label{subspaces}
  We have $f^{sos}_{G,E} \le f^{\mu}_{G,E}$ because if there exists $\lambda \in \R$ such that $f-\lambda \in \mathcal{S}_{G,E}$, i.e. $f- \lambda = \sum_i h_i^2 + g$ with $h_i \in E$ 
  and $g \in G$ then $\forall \Lambda \in \mathcal{L}_{G,E,\succcurlyeq}$, $\Lambda(f-\lambda)= \Lambda(f)- \lambda = \sum_i \Lambda(h_i^2) \ge 0.$ We deduce that $\Lambda(f) \ge \lambda$ and we conclude $f^{sos}_{G,E} \le f^{\mu}_{G,E}$.
  For the second inequality, let $\xx^*$ be a point of $\R^n$ such that $f(\xx^*)$ is the mininum of $f$ and let $\tmmathbf{1}_{x^*}\in \R[\xx]^* : p \longmapsto p(\xx^*)$ be the evaluation at $\xx^*$.
  Then we have $H_{\tmmathbf{1}_{\xx^*}}\succcurlyeq 0$ , $\tmmathbf{1}_{\xx^*}(1)=1$ and $\tmmathbf{1}_{\xx^*} (G)=0$ since $G\subset  I_{min}(f)$, so that $\tmmathbf{1}_{\xx^*} \in \mathcal{L}_{G,E,\succcurlyeq}.$ 
  We deduce the inequality $f^{\mu}_{G,E} \le f^*$.
\end{proof}

Following the relaxation approach proposed in \cite{Las01}, we are now
going to consider a hierarchy of convex optimization problems and show
that for such hierarchy, the minimum $f^{*}$ is always reached in a
finite number of steps. Let us consider the sequence of spaces
$$ 
\cdots \subset \mathcal{L}_{G,t+1,\succcurlyeq} \subset
\mathcal{L}_{G,t,\succcurlyeq} \subset \cdots 
\ \mathrm{and}\  
\cdots \subset \mathcal{S}_{G,t} \subset
\mathcal{S}_{G,t+1} \subset \cdots 
$$
for $t\in \N$, $G\subset I_{min} (f)$.
Using Remark \ref{subset} and the fact that $(G)= \cup_{t\in \N}\Span{G|t}$,
we check that 
\begin{itemize}
 \item the increasing sequence $\cdots f^{\mu}_{G,t} \le f^{\mu}_{G,t+1}
 \le \cdots \le f^*, \ for \ t \in \N$ converges to $f^{\mu}_{(G)}\le
 f^{*}$, 
  \item  the increasing sequence $\cdots f^{sos}_{G,t} \le f^{sos}_{G,t+1}
 \le \cdots \le f^*, \ for \ t \in \N$ converges to $f^{sos}_{(G)}\le f^{*}$. 
 \end{itemize}

We are going to show that these limits are attained for
some $t\in \N$. 

The next result which is a slight variation of a
result in \cite{LLR07} (and also used in \cite{lasserre:hal-00651759}) shows that for a high enough
degree, the kernel of some truncated Hankel operators allows us to
compute generators of the real radical of an ideal.

\begin{prop}\label{radical}
 For $G\subset \R[\xx]$ with $ I_{grad}(f)\subset (G)$, there exists $t_0 \in \N$ such that $\forall t\ge t_0$, 
 $\forall \Lambda \in \mathcal{L}_{G,t,\succcurlyeq},\ \sqrt[\R]{I_{grad} (f)} \subset (\ker H_{\Lambda^*}^t)$.
\end{prop}

\begin{proof}
 Let $g_1,\ldots,g_n$ be generators of $I:=I_{grad} (f), \ d_s:=
 deg(g_s),\  d:=\max_{s=1,\ldots,n} d_s$ and ${h_1,\ldots,h_k}$ be generators of the ideal $J:=\sqrt[\R]{I}$. 
By the Real Nullstellensatz, for $l \in {1,\ldots,k}$, there exist $m_l \in \N, 
 m_l \ge 1$ and polynomials $u_j^{(l)}$ and sums of squares $\sigma_l$
 such that $h_l^{2m_l} + \sigma_l= \sum_{j=1}^n u_j^{(l)}g_j$. 
As $I\subset (G)$, there exist $t_{0}'$ such that $u_j^{(l)}g_j\in \SpanD{G}{t_{0}'}$.
Set $t_0:=\max_{l\le k,j\le n}(t_{0}'$, $d,deg(h_l^{2m_l}),deg(\sigma_l))$ and let $t\ge t_0$.
 As $u_j^{(l)}g_j\in \SpanD{G}{t} $, $u_j^{(l)}g_j \in \ker
 H_{\Lambda}^t$ for all  $\Lambda \in \mathcal{L}_{G,t,\succcurlyeq}$.
Hence $h_l^{2m_l} + \sigma_l \in \ker H_{\Lambda^*}^t$, which implies
that $h_l \in \ker H_{\Lambda}^t$ since $H_{\Lambda^*}^t\succcurlyeq
0$.
\end{proof}

To compute generators of $I_{min} (f)$, we use the decomposition of $V_{\R}
(I_{grad} (f))$ in components where $f$ has a constant value.

\begin{lem}\label{pis}
 Let $V_{\R}(I_{grad} (f))= W_0 \cup W_1 \cup \ldots \cup W_s$ be the
 decomposition of the variety in disjoint real subvarieties, such that
$f(W_j)=f_j\in \R$ with $f_i < f_j \  \forall 0\le i<  j \leq s$. Then there exist polynomials $p_0,\ldots,p_r \in \R[\xx]$ 
 such that $p_i(W_j)=\delta_{ij}$, where $\delta_{ij}$ is the Kronecker delta function.
\end{lem}
\begin{proof}
As in \cite{NDS}, we decompose $V_{\C} (I_{grad} (f))$ as an union of
complex varieties \\ $V_{\C}(I_{grad} (f))= W_{\C,0} \cup W_{\C,1} \cup \ldots \cup W_{\C,s}
\cup W_{\C,s+1}$ such that $W_{\C,i}$, $i=0, \ldots, s$ have real points and
$f (W_{\C,i})=f_{i}\in \R$ is constant on $W_{\C,i}$ and $W_{\C,s+1}$ has no real point.
We number these varieties so that $f_{0}< f_{1}< \cdots< f_{s}$ and
$f_{0}=f^{*}$. By construction, the varieties $W_{i}:= W_{\C,i}\cap \R^{n}$ 
are disjoint, $f$ is constant on $W_{i}$ and $V_{\R}(I_{grad} (f))=
W_0 \cup W_1 \cup \ldots \cup W_s$. Let us take 
$p_{i} =L_{i} (f (\xx))$ where $L_{0}, \ldots, L_{s}$ are the Lagrange interpolation
polynomials at the value $f_{0}, \ldots, f_{n} \in \R$. They satisfy  $p_i(W_j)=\delta_{ij}$.
\end{proof}
 
\begin{remark}\label{idempotents}
 By definition the polynomials $p_i$ have the following properties:
 \begin{itemize}
  \item $p_0+ \ldots +p_s \equiv 1$ modulo $\sqrt[\R]{I_{grad} (f)}$.
  \item $(p_i)^2 \equiv p_i$ modulo $\sqrt[\R]{I_{grad} (f)}$ $\forall i=0,\ldots,s$.
 \end{itemize}
\end{remark}

The next results shows that in the sequence of optimization problems
that we consider, the minimum of $f$ is reached from a certain degree.

\begin{thm}\label{mayor}\label{pi-kernel}
For $G\subset \R[\xx]$ with $I_{grad} (f) \subset (G) \subset I_{min} (f)$,
there exists $t_1 \ge 0$ such that $\forall t\ge t_1$,
$f^{sos}_{G,t}=f^{\mu}_{G,t} = f^*$ and $\forall \Lambda^* \in \mathcal{L}_{G,t,\succcurlyeq}$
with $\Lambda^{*} (f)=f^{\mu}_{G,t}$, we have $p_i \in  \ker
H_{\Lambda^*}^t \ \forall i=1,\ldots,s$.
\end{thm}

\begin{proof}
By \cite{Nie11}[Theorem 2.3], there exists $t_{1}'\in \N$ such that
$\forall t\ge t_{1}'$, $f_{D,t}^{sos}= f_{D,t}^{\mu}=f^{*}$ where
$D=\{\frac{\partial f}{\partial x_1},..,\frac{\partial f}{\partial x_n}\}$. 
As $(D) = I_{grad} (f)\subset (G)= (g_{1}, \ldots, g_{s})$, there exists
$q_{i,j}\in \R[\xx]$ such that $\frac{\partial f}{\partial x_{i}}= \sum_{j=1}^{s} q_{i,j} g_{j}$. Let
$d= \max\{\deg (q_{i,j}), i=1, \ldots, n, j=1, \ldots, s\}$. Then 
$\SpanD{D}{t}\subset \SpanD{G}{t+d}$. By Remark \ref{subset} and 
Proposition \ref{minor},
for $t\geq t'_{1}$, $f^{*}=f^{sos}_{D,t}\leq f^{sos}_{G,t+d}\leq
f^{\mu}_{G,t+d}\leq f^{*}$. Thus for $t\geq t'_{1}+d$, we have
\begin{equation}\label{desco0}
 f^{sos}_{G,t}= f^{\mu}_{G,t}=f^{*}.
\end{equation}

Let $J =\sqrt[\R]{I_{grad} (f)}$. 
By Lemma \ref{pis} and Remark \ref{idempotents}, we can write 
 \begin{equation}
  f\equiv \sum_{i=0}^s f_i\,p_i^2 \ modulo \ J,  \label{desco1}
 \end{equation}
 where $f_i=f(W_i) \in \R$ and $f_0=f(W_0)=f^*$ .
 Then
 \begin{equation}
  f\equiv f^* p_0^2 + \sum_{i=1}^s f_i \,p_i^2 \equiv f^*(1-\sum_{r=1}^s p_r^2)+\sum_{i=1}^s f_i \,p_i^2 \equiv f^*+\sum_{i=1}^s (f_i-f^*)\,p_i^2 \ modulo \ J. \label{desco2}
 \end{equation}
 Hence
 \begin{equation}
  f-f^*\equiv \sum_{i=1}^s (f_i-f^*)\,p_i^2 + h. \label{desco3}
 \end{equation}
with $h\in J$.
By Theorem \ref{radical}, there exists $t''_{1}\geq t_{0}$ such that $\Lambda
(h)=0$ for all $\Lambda \in \ker \mathcal{L}_{G,t_{1}'',\succcurlyeq}$. 

Let us take $t_1 = max\{ t_{1}'+d, t_{1}'', \deg(p_1), \ldots, \deg (p_{s})\}$,  $t\ge
t_{1}$ and $\Lambda^{*} \in \mathcal{L}_{G,t,\succcurlyeq}$ such that 
$\Lambda^{*} (f)=f^{\mu}_{G,t}$.
From Equation \eqref{desco0}, we deduce that
\begin{equation}\label{desco4}
\Lambda^{*} (f-f^{*}) = 0  
=  \sum_{i=1}^s (f_i-f^*) \Lambda^{*} (p_i^2) 
\end{equation} 
This implies that 
$\Lambda^*(p_i^2) = 0$ and $p_{i} \in \ker H_{\Lambda^{*}}^{t}$ for $i=1, \ldots, s$, since $f_i - f^* > 0$ and
$ \Lambda^{*} \succcurlyeq 0$ on $\R[\xx]_t$.
\end{proof}

The example 3.4 in \cite{NDS} shows that it may not always be possible to write
$f-f^{*}$ as a sum of squares modulo $I_{grad} (f)$ but, from the
previous result we see that $f-f^{*}$ is the limit of sums of squares modulo $I_{grad}
(f)\cap \R[\xx]_{t}$ for a fixed $t\geq t_{1}$. Under the assumption that
there exists $\xx^{*}\in \R^{n}$ such that $f (\xx^{*})=f^{*}$, we can
construct $\Lambda^{*}=\ev_{\xx^{*}}\in \Lc_{G,t,\succcurlyeq}$ such
that $\Lambda^{*} (f) = f^{*}$. This means that $f^{\mu}_{G,t}$ is
reached for $t\geq t_{1}$.

A direct corollary of the previous theorem is that for any $G\subset \R[\xx]$ with $I_{grad}
(f) \subset (G) \subset I_{min} (f)$, we have
$$ 
f^{sos}_{I_{grad} (f)} = f^{sos}_{(G)}= f^{\mu}_{I_{grad} (f)}=f^{\mu}_{(G)}=f^{*}.
$$

As for the construction of generators of the real radical
$\sqrt[\R]{I_{grad} (f)}$ (Proposition \ref{radical}), we can
construct generators of $I_{min} (f)$ from the kernel of a truncated Hankel operator associated to any linear form which minimizes $f$:
\begin{thm}\label{minker}
For $G\subset \R[\xx]$ with $I_{grad} (f) \subset (G) \subset I_{min} (f)$,
there exists $t_2 \in \N$ such that
$\forall t\ge t_2$, for $\Lambda^* \in \mathcal{L}_{G,t,\succcurlyeq}$
with $\Lambda^{*} (f)=f^{\mu}_{G,t}$, we have $I_{min}(f) \subset (\ker H_{\Lambda^*}^t)$.
\end{thm}

\begin{proof}
 Let $I=I_{grad} (f)$ and $J =\sqrt[\R]{I}$. 
 We consider again the above decomposition $V_{\R}(I)= W_0 \cup W_1 \cup
 \ldots \cup W_s$
with $f_{0}=f (W_{0})< f_{1} = f (W_{1}) < \cdots< f_{s}= f (W_{s})$.
We denote by $h_{1}, \ldots, h_{k}$ a family of generators of the ideal
$I_{min}(f)=I(W_0)$ and $d=\max\{\deg (h_{1}),\ldots, \deg (h_{k})\}$.

Let us fix $0\le j\le k$ and show that $h_{j} \in \ker H_{\Lambda^*}^t$ for $t$
sufficiently large. 
We know that $p_0 h_{j}^2 \mid_{W_0} =0$ and that for $i=1,\ldots,s$, 
 $p_0 h_{j}^2 \mid_{W_i} =0 $ which implies that $p_0
 h_{j}^2 \in J:= \sqrt[\R]{I_{grad} (f)}$. By Theorem \ref{radical}, there
 exists $t_{2,j}'$ such that $\Lambda (p_0 h_{j}^2)=0$  for all
 $\Lambda \in \mathcal{L}_{G,t,\succcurlyeq}$
with $t\ge t_{2,j}'$.

By Theorem \ref{pi-kernel}, $p_i \in \ker H_{\Lambda^*}^t$ for $t\geq
t_{1}$. By Remark \ref{idempotents}, $p_0+ \ldots +p_s \equiv 1$
modulo $J$.
 By Theorem \ref{radical}, there
 exists $t_{2}''$ such that $\Lambda (p_0+ \ldots +p_s-1)=0$  for all
 $\Lambda \in \mathcal{L}_{G,t,\succcurlyeq}$
with $t\ge t_{2}''$.

Let us take $t_{2} := \max\{t_{1}+2\, d, t_{2}''+ 2\,d, t_{2,1}', \ldots, t_{2,k}'\}$ and
$t\ge t_{2}$. Then for
$\Lambda^* \in \mathcal{L}_{G,t,\succcurlyeq}$ with $\Lambda^{*} (f)= f^{\mu}_{G,t}$,
we have 

$$
\Lambda^*(h_{j}^2) = \Lambda^*(h_{j}^2 p_0)+\Lambda^*(h_{j}^2
p_1)+\ldots+\Lambda^*(h_{j}^2 p_s) =0.
$$
Hence, $h_{j} \in \ker H_{\Lambda^*}^t$ for $j=1, \ldots, k$ and $I_{min}(f) \subset (\ker H_{\Lambda^*}^t)$.
\end{proof}

We introduce now the notion of {\em generic linear form for $f\in \R[\xx]$}. Such
a linear form will allow us to compute $I_{min} (f)$ as we will see.
\begin{prop}\label{generic}
 For $\Lambda^* \in \mathcal{L}_{G,E,\succcurlyeq}$ and $p \in \R[\xx]$, the following assertions are equivalent:
 \begin{itemize}
  \item[(i)] $rank H_{\Lambda^*}^{E} = max_{\Lambda \in \mathcal{L}_{G,E,\succcurlyeq},\Lambda(p)=p^{\mu}_{G,E}} rank H_{\Lambda}^{E}$.
  \item[(ii)]  $\forall \Lambda \in \mathcal{L}_{G,E,\succcurlyeq}$
    such that $\Lambda(p)=p^{\mu}_{G,E}$, $\ker H_{\Lambda^*}^{E} \subset \ker H_{\Lambda}^{E}$.
  \item[(iii)]$rank H_{\Lambda^*}^{E_0} = max_{\Lambda \in \mathcal{L}_{G,E,\succcurlyeq}, \Lambda(p)=p^{\mu}_{G,E}} rank H_{\Lambda}^{E_0}$ for any subspace $E_0 \subset E$.
 \end{itemize}
 We say that  $\Lambda^* \in \mathcal{L}_{G,E,\succcurlyeq}$ is
 generic for $p$ if it satisfies one of the equivalent conditions (i)-(iii).
 \end{prop}
 
 \begin{proof}
(i) $\Longrightarrow$ (ii): Note that $\frac{1}{2}(\Lambda+\Lambda^*)
\in \mathcal{L}_{G,E,\succcurlyeq}$ and $\frac{1}{2}(\Lambda+\Lambda^*)(p)=
p^{\mu}_{G,E}$
and $\ker H^E_{\frac{1}{2}(\Lambda+\Lambda^*)}=\ker H^E_\Lambda \cap \ker H^E_{\Lambda^{*}}$ (using Lemma \ref{positive}). Hence,
$rank  H^E_{\frac{1}{2}(\Lambda+\Lambda^*)} \ge rank  H^E_{\Lambda^{*}}$ and
thus equality holds. This implies that $\ker H^E_{\frac{1}{2}(\Lambda+\Lambda^*)}=\ker H^E_{\Lambda^{*}}$ is thus contained in 
$\ker H^E_\Lambda$.\\
(ii) $\Longrightarrow$ (iii): Given $E_0\subset E$, we show that
$\ker H^{E_0}_{\Lambda^{*}}\subset \ker H^{E_0}_{\Lambda}$. By Lemma \ref{positive}, we have
$\ker H^{E_0}_{\Lambda^{*}}\subset \ker H^{E}_{\Lambda^{*}}$ and, by the above, we have
$\ker H^{E}_{\Lambda^{*}} \subset \ker H^E_\Lambda$.\\
(iii) $\Longrightarrow$ (i) This implication is obvious.
\end{proof}

The next result, which refines Theorem \ref{minker}, shows only
elements in $I_{min} (f)$ are involved in the kernel of a truncated
Hankel operator associated to a generic linear form for $f$.
\begin{thm}\label{kermin}
Let $E, G$ be as in Definition \ref{mintruncated} with $G\subset I_{min}
(f)$. If $\Lambda^* \in \mathcal{L}_{G,E,\succcurlyeq}$ 
is generic for $f$ and such that $\Lambda^{*} (f) = f^{*}$, then $\ker
H_{\Lambda^*}^{E} \subset I_{min}(f)$.  
\end{thm}

\begin{proof}
 Let $\xx^* \in \R^n$ such that $f(\xx^*)=f^*$ is minimum. Let $\underline{\mathbf{1}}_{\xx^*}$ denotes the evaluation at $\xx^*$ restricted to
 $\Span{E\cdot E}$ and $h \in \ker H_{\Lambda^*}^E$. Our objective is to show that $h(\xx^*)=0$.
 Suppose for contradiction that $h(\xx^*) \neq 0$. We know that
 $\underline{\mathbf{1}}_{\xx^*}\in \mathcal{L}_{G,E,\succcurlyeq}$
 since $G\subset I_{min} (f)$
 and $\underline{\mathbf{1}}_{\xx^*}(f)=f(\xx^*)=f^*$. We define $\tilde{\Lambda}=\frac{1}{2}(\Lambda^*+\underline{\mathbf{1}}_{\xx^*})$.
 By definition $\tilde{\Lambda} \in \mathcal{L}_{G,E,\succcurlyeq}$ and 
 $\tilde{\Lambda}(f)=\frac{1}{2}(\Lambda^*(f)+\underline{\mathbf{1}}_{\xx^*}(f))=\frac{1}{2}(\Lambda^*(f)+f(\xx^*))=f^*$.
 As $h \in \ker H_{\Lambda^*}^E$, $$\tilde{\Lambda}(h^2)=\frac{1}{2}(\Lambda^*(h^2)+\underline{\mathbf{1}}_{\xx^*}(h^2))=\frac{1}{2}h^2(\xx^*)\neq 0$$ 
 thus $h \in \ker H_{\Lambda^*}^E \setminus \ker H_{\tilde{\Lambda}}^E$ 
 and by the maximality of the rank of $H_{\Lambda^*}^E$, $\ker H_{\tilde{\Lambda}}^E \not\subset \ker H_{\Lambda^*}^E$. 
 Hence there exits $\tilde{h} \in \ker H_{\tilde{\Lambda}}^E \setminus \ker H_{\Lambda^*}^E$.
 Then $0= H_{\tilde{\Lambda}}^E(\tilde{h})=\frac{1}{2}(H_{\Lambda^*}^E(\tilde{h})+H_{\underline{\mathbf{1}}_{\xx^*}}^E(\tilde{h}))=
 \frac{1}{2}(H_{\Lambda^*}^E(\tilde{h})+\tilde{h}(\xx^*) \cdot \underline{\mathbf{1}}_{\xx^*})$. As $H_{\Lambda^*}^E(\tilde{h}) \neq 0$
 implies $\tilde{h}(\xx^*) \neq 0$. On the other hand,$$0= H_{\tilde{\Lambda}}^E(\tilde{h})(h)=\tilde{\Lambda}(h\tilde{h})
 =\frac{1}{2}(\Lambda^*(h\tilde{h})+h(\xx^*)\tilde{h}(\xx^*))=\frac{1}{2}(H_{\Lambda^*}^E(h)(\tilde{h})+h(\xx^*)\tilde{h}(\xx^*)).$$
 As $h \in H_{\Lambda^*}^E$, then $$0= H_{\tilde{\Lambda}}^E(\tilde{h})(h)=\frac{1}{2}(h(\xx^*)\tilde{h}(\xx^*)).$$ 
 As $\tilde{h}(\xx^*) \neq 0$, since we have supposed that $h(\xx^*) \neq 0$ it yields a contradiction.
\end{proof}

The last result of this section shows that a generic linear form for $f$
yields the generators of the minimizer ideal $I_{min} (f)$ in high enough degree.
\begin{thm}\label{thprin}
For $G\subset \R[\xx]$ with $I_{grad} (f) \subset (G) \subset I_{min} (f)$,
there exists $t_2\in \N$ (defined in Theorem \ref{minker}) such that
such that $\forall t\ge t_2$, for $\Lambda^* \in
\mathcal{L}_{G,t,\succcurlyeq}$ generic for $f$, we have $\Lambda^*(f)=f^*$ and $(\ker H_{\Lambda^*}^t)=I_{min}(f)$.
\end{thm}

\begin{proof}
We obtain the result as consequence of Theorem \ref{mayor}, Theorem \ref{minker} and Theorem \ref{kermin}.
\end{proof}
As a consequence, in a finite number of steps, the sequence of
optimization problems that we consider gives the minimum of $f$ and
the generators of $I_{min} (f)$.

\section{Zero-dimensional Case}\label{sec-0dim}
In this section, we describe a criterion to detect when the kernel of
a truncated Hankel operator associated to a generic linear form for
$f$ yields the generators of the minimizer ideal.  It is based on a
flat extension property \cite{MoLa2008} and applies to global polynomial optimization
problems where the minimizer ideal $I_{min} (f)$ is zero-dimensional.

\begin{defn} \label{defflatextension}
Given vector subspaces $E_0 \subset E \subset \K[\xx]$ and $\Lambda \in { \Span{E \cdot E}}^*$, $H_{\Lambda}^E$ is
said to be a \textit{flat extension} of its restriction $H^{E_0}_{\Lambda}$ if $rank H^E_{\Lambda} = rank H^{E_0}_{\Lambda}$.
\end{defn}

We recall here a result from \cite{MoLa2008}, which gives a rank condition
for the existence of a flat extension of a truncated Hankel operator.

\begin{thm}\label{theoflatextension}
Consider a monomial set $B\subset \Mon$ connected to 1 and a linear
function $\Lambda$ defined on $\Span{B^+ \cdot B^+}$. 
Let $E=\Span{B}$ and $E^{+}=\Span{B^{+}}$.
Assume that $\rank\, H^{E^+}_{\Lambda} = \rank\, H^E_{\Lambda} = |B|$. 
Then there exists a (unique) linear form $\tilde{\Lambda} \in \K[\xx]^*$ which extends
 $\Lambda$, i.e., $\tilde{\Lambda} (p) = \Lambda (p)$ for all $p \in
 \Span{B^+ \cdot B^+}$, 
satisfying $\rank\, H_{\tilde{\Lambda}} = \rank\,
 H^{E^+}_{\Lambda}$. Moreover, we have $\ker H_{\tilde{\Lambda}}= (\ker H^{E^{+}}_{\Lambda})$.
\end{thm}
In other words, the condition $\rank\, H^{E^+}_{\Lambda} = \rank\,
H^E_{\Lambda} = |B|$ 
implies that the truncated Hankel operator $H^{B^+}_{\Lambda}$ has
a (unique) flat extension to a (full) Hankel operator
$H_{\tilde{\Lambda}}$.

\begin{prop}\label{general}
Let $E$, $G$ be as in Definition \ref{mintruncated} with $G \subset
I_{min} (f)$.
If $\Lambda^* \in \mathcal{L}_{G,E,\succcurlyeq}$ 
coincides with a probability measure $\mu$ on $\langle E \cdot E
\rangle$ and satisfies $\Lambda^*(f)= f^{\mu}_{G,E}$, then $\Lambda^*(f)= f^{\mu}_{G,E}=f^*$.
\end{prop}
\begin{proof}
By Proposition \ref{minor}, $f^{\mu}_{G,E} \le f^*$. Conversely as
$f(x) \ge f^*$ for all $x \in \R^n$, we have $\Lambda^*(f) = \int f d\mu \ge \int f^* d\mu=  f^*$. 
\end{proof}

\begin{prop}\label{nogap}
If there exists $\Lambda \in \mathcal{L}_{G,E,\succcurlyeq}$ with $\ker H_{\Lambda}^E = \{ 0 \}$, then $f^{sos}_{G,E} = f^{\mu}_{G,E}$.
\end{prop}
\begin{proof}
  If $\ker H_{\Lambda}^E = \{ 0 \}$ implies that $\Lambda \succ 0$,
  i.e, $H_{\Lambda}^E > 0$. Hence by Slater's Theorem of
  \cite{Boyd-Vandenbergh} we have the strong duality then
  $f^{sos}_{G,E} = f^{\mu}_{G,E}$.
\end{proof}

\begin{thm}\label{caszero}
Let $B$ is a monomial set connected to $1$, $E=\Span{B^{+}}$ and 
$G\subset \Span{B^{+}\cdot B^{+}} \cap I_{min} (f)$.
Let $\Lambda^*\in \mathcal{L}_{G,E,\succcurlyeq}$ such that 
$\Lambda^{*}$ is generic for $f$ and satisfies the flat extension property: 
$\rank\, H_{\Lambda^*}^{B^{+}} = \rank\, H_{\Lambda^*}^{B}= | B |$. 
Then there is no duality gap, $f^*=f^{\mu}_{G,E}=f^{sos}_{G,E}$ and
$(\ker H_{\Lambda^*}^{B^+}) = I_{min}(f).$
\end{thm}

\begin{proof}
As $\rank\, H_{\Lambda^*}^{B^{+}} = \rank\, H_{\Lambda^*}^{B}= | B|$,
Theorem \ref{theoflatextension} implies that there exists a (unique)
linear function $\tilde{\Lambda}^{*} \in \K[\xx]^*$ which extends
$\Lambda^{*}$.
As $\rank\, H_{\tilde{\Lambda}^{*}} = \rank\,
 H^{B}_{\Lambda^{*}}=|B|$ and $\ker H_{\tilde{\Lambda}^{*}}= (\ker
 H^{B^{+}}_{\Lambda})$, any polynomial $p\in \R[\xx]$ can be reduced modulo
 $\ker H_{\tilde{\Lambda}^{*}}$ to a polynomial $b\in \Span{B}$ so
 that $p-b\in \ker H_{\tilde{\Lambda}^{*}}$.
Then $\tilde{\Lambda}^{*} (p^{2})= \tilde{\Lambda}^{*} (b^{2}) =
\Lambda^{*} (b^{2}) \geq 0$ since $\Lambda^*\in \mathcal{L}_{G,E,\succcurlyeq}$.
This implies that $\tilde{\Lambda}^* \succcurlyeq 0$. By Proposition
\ref{proppositive}, 
$\tilde{\Lambda}^*$ has a decomposition
$\tilde{\Lambda}^*=\sum_{i=1}^r \lambda_i\ev_{\zeta_i}$  
with $\lambda_i > 0$ and $\zeta_i \in \R^n$.

As $\tilde{\Lambda}^{*} (1)={\Lambda}^{*} (1)=1$,  $\tilde{\Lambda}^{*}$ is a probability measure.
By Proposition \ref{general},  $\tilde{\Lambda}^*(f)={\Lambda}^*(f)= f^{\mu}_{G,E}=f^{*}$.

As $\tilde{\Lambda}^*=\sum_{i=1}^r \lambda_i\tmmathbf{1}_{\zeta_i}$  
with $\lambda_i > 0$ and $\zeta_i \in \R^n$ and as
$\tilde{\Lambda}^* (1)=\sum_{i=1}^{r}\lambda_{i}=1$ and 
$\tilde{\Lambda}^* (f)=\sum_{i=1}^{r}\lambda_{i} f (\zeta_{i}) = f^{*}$, we deduce that 
$f(\zeta_i)=f^*$ for $i=1, \ldots, r$ so that $\{\zeta_{1}, \ldots,
\zeta_{r}\}\subset V (I_{min} (f))$.

By Proposition \ref{propradideal}, Theorem
\ref{theoflatextension} and Theorem \ref{kermin}, we have
$$ 
\ker H_{\tilde{\Lambda}^{*}}
 = I(\zeta_{1}, \ldots, \zeta_{r}) = (\ker H^{B^{+}}_{\Lambda^{*}}) \subset I_{min} (f).
$$
We deduce that $\{\zeta_{1}, \ldots, \zeta_{r}\} = V (I_{min} (f))$
so that $(\ker H^{B^{+}}_{\Lambda^{*}})= I_{min} (f)$.

As we have the flat extension condition, $ker H_{\Lambda^*}^{B}
=\{0\}$ and by Proposition \ref{nogap} there is not duality gap:
$f^*=f^{\mu}_{G,E}=f^{sos}_{G,E}$.
\end{proof}
Notice that if the hypotheses of this theorem are satisfied, then
necessarily $I_{min} (f)$ is zero-dimensional.

\section{Minimizer border basis algorithm}
In this section we describe the algorithm to compute the global
minimum of a polynomial, assuming $f^{*}$ is reached in $\R^{n}$ and
$I_{min} (f)$ is zero dimensional.
It can be seen as a type 
of border basis algorithm, for we insert an additional step in the
main loop. It is closely connected to the real radical border basis
algorithm presented in \cite{lasserre:hal-00651759} but instead of
 ``minimizing zero'' to generate new elements in the real radical, we
minimize $f$ to compute generators of the minimizer ideal $I_{min} (f)$.

\subsection{Description}
The convex optimization problems that we consider are the following:\\
\begin{algorithm2e}[H]\caption{\textsc{Optimal Linear
      Form}}
\noindent{}\textbf{Input:} $f\in \R[\xx]$, $M=(\xx^{\alpha})_{\alpha \in A}$ a monomial set containing $1$
with $f= \sum_{\alpha \in A+ A} f_{\alpha} \xx^{\alpha} \in
\Span{M\cdot M}$, $G\subset \R[\xx]$.\\
\noindent{}\textbf{Output:} 
the minimum $f^{*}_{G,M}$ of $\sum_{\alpha\in A+ A} \lambda_{\alpha} f_{\alpha}$ subject  to:
     \begin{itemize}
        \item[--] $H_{\Lambda^*}^M= (h_{\alpha,\beta})_{\alpha,\beta\in A} \succcurlyeq 0$,
        \item[--] $H_{\Lambda^*}^M$ satisfies the Hankel constraints
               $h_{0,0}=1$, and $h_{\alpha,\beta}=h_{\alpha',\beta'}= \lambda_{\alpha+\beta}$
               \\if $\alpha+\beta=\alpha'+\beta'$,
       \item[--]  $\Lambda^{*} (g)= \sum_{\alpha \in  A+
            A}\, g_{\alpha} \lambda_{\alpha}=0$ for all
          $g=\sum_{\alpha \in  A+ A}\, g_{\alpha} \xx^{\alpha}
          \in G\cap \Span{M\cdot M}$.
     \end{itemize}
and $\Lambda^{*}\in \Span{M \cdot M}^{*}$
represented by the vector $[\lambda_{\alpha}]_{\alpha \in A + A}$.
\end{algorithm2e}
This optimization problem is a Semi-Definite Programming problem,
corresponding to the optimization of a linear functional on a
linear subspace of the cone of Positive Semi-Definite matrices.
It is a convex optimization problem, which can be solved efficiently
%(in polynomial time up to some precision \cite{XXX}) 
by SDP solvers.
If an Interior Point Methods is used, the solution $\Lambda^{*}$ is in the interior of
a face on which the minimum $\Lambda^{*} (f)$ is reached so that $\Lambda^{*}$ is generic for $f$.
This is the case for tools such as \texttt{csdp} or \texttt{sdpa},
that we will use in the experimentations.

In the case, where $M=B$ is a monomial set connected to $1$, $F$ is a complete rewriting family for $B$ in degree $\leq
2t$, and $G= \{b-\pi_{F,B} (b); b \in B\cdot B\}$, we will also denote 
\textsc{Optimal Linear Form} $(f,B_{t},\pi_{F,B})$:=\textsc{Optimal Linear Form} $(f,B_{t}, G)$.

\begin{algorithm2e}[H]
\caption{\textsc{Minimizer Ideal of $f$}}
\KwIn {A real polynomial function $f$ with $I_{min}(f)\neq (1)$ and zero-dimensional.\\}
$F:=\{ \frac{\partial   f}{\partial x_1},..,\frac{\partial f}{\partial
  x_n}\}$; 
$t:=\lceil deg(f)/2 \rceil$; 
$B$ := set of monomials of degree $\leq t$;\\
$\tilde{f}:=- \infty$; stop:=false; \\

While not stop\\ 
\begin{enumerate}
 \item Compute the commuting relations for $F$ with respect to $B$ in
   degree $\le 2\,t$.
 \item Reduce them by the existing relations $F$.
 \item Add the non-zero reduced relations to $F$ and update $B$.
 \item Let $[f^{*}_{F,B_{t}}, \Lambda^{*}]:= \textsc{Optimal Linear Form} (f,B_{t},\pi_{F,B_{t}})$
 %  \begin{itemize}
 %   \item If duality gap $\Rightarrow$ go to 1 with $t:=t+1$.
 %   \item If no duality gap $\Rightarrow$ let $[\Lambda_{t+1}^{*}, f^{\mu}_{F,B_{t+1}}]:= \textsc{Optimal Linear Form} (f,B_{t+1},F_{2\,t+2})$.
 %     \begin{itemize}
 %      \item If $f^{\mu}_{F,B_t}<f^{\mu}_{F,B_{t+1}} \Rightarrow$ go to 1 with $t:=t+1$.
 %      \item If $f^{\mu}_{F,B_t}=f^{\mu}_{F,B_{t+1}}$ $\Rightarrow$ 
 %      \begin{itemize}
 % \item Compute the border basis $F'$ of $F + \ker H_{\Lambda_{t+1}^*}^{B_{t+1}}$,\\
 %     and the basis $B'$.
 % \item $[\Lambda', f^{\mu}_{F',B'_{t}}] := \textsc{Optimal Linear Form} (f,B'_{t},F')$.
 %       \item If $f^{\mu}_{F',B'}=f^{\mu}_{F,B_{t}}$ and $\ker
 %         H_{\Lambda'}^{B'_t}=\{0\}  \Rightarrow$ stop=true, $F=F'$,
 %         $B=B'$, $f^{*}= f^{\mu}_{F,B_{t}}$.
 %       \item[] else go to 1  with $t:=t+1$.
 %      \end{itemize}
 %     \end{itemize}
 %    \end{itemize}  
      \begin{enumerate}
\item If there is a duality gap then go to 1 with $t:=t+1$.
\item If $\tilde{f}\neq f^{*}_{F,B_{t}}$ then $\tilde{f}:=
  f^{*}_{F,B_{t}}$; go to 1 with $t:=t+1$.
\item Compute the border basis $F'$ of $F + \ker
 H_{\Lambda^*}^{B_{t-1}}$ and the basis $B'$ in degree $\leq 2\,(t-1)$.
 \item If there exists $b'\in B'$ with $\deg (b')\geq t-1$ then go to (1)  with $t:=t+1$.   
 \item Let $[f^{*}_{F',B'}, \Lambda'] := \textsc{Optimal Linear Form} (f,B',\pi_{F',B'})$.
\item If $f^{*}_{F',B'}=\tilde{f}$ and $\ker
        H_{\Lambda'}^{B'_{t-1}}=\{0\}  \Rightarrow$ stop=true, $F=F'$,
         $B=B'$, $\tilde{f}= f^{*}_{F,B_{t-1}}$.
 \item[] else go to 1  with $t:=t+1$.   
      \end{enumerate}
 \end{enumerate}
\KwOut{A basis $B$ of ${\mathcal A} = \R[\xx]/I_{min}(f)$, a border
  basis $F$ of $I_{min} (f)$ for $B$ and the minimum $\tilde{f}$.}
\end{algorithm2e}

\subsection{Correctness of the algorithm}
In this subsection, we analyse the correctness of the algorithm.

\begin{lem}\label{lemextend}
Let $t\in \N$, $B\subset \R[\xx]_{2t}$ be a monomial set connected to $1$,
$F\subset \R[\xx]$ be a border basis for $B$ in degree $\leq 2t$ and
 $G=\Span{B_{t}\cdot B_{t}}\cap \SpanD{F}{2t}$.
Let $E\subset \R[\xx]_{t}$ be a vector space containing $B^{+}\cap \R[\xx]_{t}$
and $G'=\Span{E\cdot E}\cap \SpanD{F}{2t}$. 
For all
$\Lambda \in \Lc_{G,B_{t},\succcurlyeq}$, there exists a unique
$\tilde{\Lambda} \in \Lc_{G',E,\succcurlyeq}$ which extends $\Lambda$.
Moreover, $\tilde{\Lambda}$ satisfies
$\rank\, H_{\tilde{\Lambda}}^{E} = \rank\, H_{{\Lambda}}^{B_{t}}$ and
$\ker H_{\tilde{\Lambda}}^{E} = \ker H_{{\Lambda}}^{B_{t}} + \SpanD{F}{t}\cap E$.
\end{lem}
\begin{proof}
Suppose that $F\subset \R[\xx]$ be a border basis for $B$ in degree $\leq 2t$, that
is, all boundary polynomials of $C^+(F_{2\,t})$ reduces to $0$ by
 $F_{2\,t}$. Then by Theorem \ref{thmnfdegd}, we have
 $\R[\xx]_{2\,t}=\Span{B}_{2\,t} \oplus \SpanD{F}{2\,t}$ and
 $\Span{E\cdot E}=\Span{B}_{2\,t} \oplus  G'$,
 $\Span{B_{t}\cdot B_{t}}=\Span{B_{2\,t} \cap B_{t}\cdot B_{t}} \oplus  G$.
Thus for all $\Lambda\in \Lc_{G,B_{t},\succcurlyeq}$,
there exists a unique $\tilde{\Lambda}\in  \Span{E\cdot E}^{*}$
such that $\tilde{\Lambda}(G')=0$ and 
$\tilde{\Lambda} (b) = \Lambda (b)$ for all $b\in B_{t}\cdot B_{t}$. 

As $E\cdot (E\cap \SpanD{F}{t}) \subset 
\Span{E\cdot E}\cap \SpanD{F}{2t}=G'$, we have 
$\tilde{\Lambda} ( E\cdot  (\SpanD{F}{t}\cap E) ) =0$ so that 
\begin{equation}\label{eq:incl1}
\SpanD{F}{t}\cap E \subset \ker H_{\tilde{\Lambda}}^{E}.
\end{equation}

For any element $b\in \ker H_{{\Lambda}}^{B_{t}}$ we have
$\forall b'\in B_{t}, \Lambda (b\,b')=\tilde{\Lambda} (b\,b') = 0$. As
$\tilde{\Lambda} (B_{t}\cdot (\SpanD{F}{t}\cap E))=0$ and $E= \Span{B_{t}} \oplus
(\SpanD{F}{t}\cap E)$, for any element $e\in E$, $\tilde{\Lambda}
(b\,e) = 0$. This proves that 

\begin{equation}\label{eq:incl2}
\ker H_{{\Lambda}}^{B_{t}}\subset \ker H_{\tilde{\Lambda}}^{E}.
\end{equation}
 
Conversely as $E= \Span{B_{t}} \oplus (\SpanD{F}{t}\cap E)$, any element
of $E$ can be reduced modulo $\SpanD{F}{t}\cap E$ to an
element of $\Span{B_{t}}$, which shows that 
\begin{equation}\label{eq:incl3}
\ker H_{\tilde{\Lambda}}^{E} \subset \ker H_{{\Lambda}}^{B_{t}} +
\SpanD{F}{t}\cap E.
\end{equation}

From the inclusions \eqref{eq:incl1}, \eqref{eq:incl2} and \eqref{eq:incl3}, we deduce that $\ker
H_{\tilde{\Lambda}}^{E} = \ker H_{{\Lambda}}^{B_{t}} +
(\SpanD{F}{t}\cap E)$
and that $\rank\, H_{\tilde{\Lambda}}^{E} = \rank\, H_{{\Lambda}}^{B_{t}}$.

As $\Lambda \succcurlyeq 0$,
 by projection along $(\SpanD{F}{t}\cap E)\subset 
\ker H_{\tilde{\Lambda}}^{E}$, 
$\forall p\in E$ there exists $b\in \Span{B_{t}}$ such that
$\tilde{\Lambda} (p^{2})= \Lambda (b^{2})\geq 0$. 
Thus $\tilde{\Lambda}\succcurlyeq 0$ which ends the proof of this lemma.
\end{proof}

\begin{lem}\label{algostopbien}
If the algorithm terminates, then $F'$ is a border basis of $I_{min}
(f)$ for $B'$.
\end{lem}
\begin{proof}
If the algorithme stops, this can happen only in step (4) in some degree
$t+1$ such that 
\begin{itemize}
 \item $[f^{*}_{F,B_{t+1}}, \Lambda^{*}]:= \textsc{Optimal Linear Form} (f,B_{t+1},F)$,
 \item $F'$ is a border basis of $F + \ker H_{\Lambda^*}^{B_{t}}$ for $B'$ in degree $\leq 2\,t$,
 \item the monomials in $B'$ are 
of degree $< t$ so that $B'_{t}=B'$,
 \item $[f^{*}_{F',B'}, \Lambda'] := \textsc{Optimal Linear Form} (f,B',F')$ 
with $\tilde{f}=\Lambda' (f)= f^{*}_{F',B'}= \Lambda^{*}(f) =
f^{*}_{F,B_{t}}$ and $\ker  H_{\Lambda'}^{B'}=\{0\}$.
\end{itemize}
Then $\Span{B'^{+}} = \Span{B'}\oplus \Span{F'}$. By Lemma
\ref{lemextend} with $E=\Span{B'^{+}}$, the linear form
$\Lambda' \in \Lc_{F',B',\succcurlyeq}$ can be extended to a linear
form ${\Lambda}'' \in \Lc_{F'',B'^{+},\succcurlyeq}$ with
$F''=\SpanD{F'}{2t}\cap \Span{B'^{+}\cdot B'^{+}}$ such that $\ker
H_{{\Lambda}''}^{B'^{+}}= \SpanD{F'}{t}\cap \Span{B^{+}}= \Span{F'}$. As $\ker
H_{\Lambda'}^{B'}=\{0\}$ 
and $\Span{B'^{+}} = \Span{B'}\oplus \Span{F'}$,
we have $\rank\, H_{{\Lambda}''}^{B'^{+}} = \rank\, H_{\Lambda'}^{B'}=|B'|$
and the flat extension theorem (Theorem \ref{theoflatextension})
applies: $\Lambda''$ is the restriction to $\Span{B'\cdot B'}$
of a positive linear form 
$$ 
{\tilde{\Lambda}}'= \sum_{i=1}^{r} \lambda_{i}\, \ev_{\zeta_{i}}
$$
with $r=|B'|$, $\zeta_{i}\in \R^{n}$ distinct, $\lambda_{i}>0$ and
$\sum_{i=1}^{r} \lambda_{i}=1$. 
Then $\Lambda' (f) =\sum_{i=1}^{r} \lambda_{i} f (\zeta_{i}) \geq
f^{*} $. By hypothesis, 
$\Lambda' (f) = f^{*}_{F_{2\,t},B_{t}} \le f^{*}$ since $F\subset I_{grad} (f)
\subset I_{min} (f)$. We deduce 
that $\Lambda' (f) = \Lambda^{*} (f) = f^{*}$, 
that $(F_{2t}+ \ker H_{\Lambda^{*}}^{B_{t}})= (F') \subset I_{min}(f)$ by Theorem \ref{kermin}
and that $(F')= I_{min}(f)$ by Theorem \ref{caszero}.
Therefore $F'$ is a border basis of $I_{min}(f)$ for $B'$. 
\end{proof}
 
\begin{prop}
 Assume that $I_{min}(f)$ is zero-dimensional. Then the algorithm terminates. It outputs a border basis $F$ for $B$ connected to 1, such that $\R[\xx]=\< B \> \oplus (F)$ 
 and $(F)=I_{min}(f)$.
\end{prop}
\begin{proof}
 First, we are going to prove by contradiction that when $I_{min}(f)$ is zero-dimensional, the algorithm terminates.
 Suppose that the loop goes for ever. Notice that at each step either $F$ is extended by adding new linearly independent polynomials or it moves to $t+1$. Since the number of 
 linearly independent polynomials added to $F$ in degree $\le 2\, t$
 is finite, there is a step in the loop from which $F$ is not modified
 any more in degree $\le$ 2\, t. In this case, all boundary
 C-polynomials of elements of $F$ of degree $\le 2\, t$ are reduced to $0$ by
 $F_{2\, t}$
and $F_{2t}$ is a border basis for $B_{2t}$ in degree $\leq 2t$. 
By Lemma \ref{lemextend} with $E=R_{t}$, $\Lambda^{*}$ extends to a
 linear form $\tilde{\Lambda}^{*}\in \R[\xx]_{2\, t}^{*}$ such that
$$ 
K^{t}:=\ker H_{\tilde{\Lambda}^{*}}^{t} = \ker H_{\Lambda^{*}}^{B_{t}}+\SpanD{F}{t}.
$$
By Theorem \ref{thprin}, for $t\ge t_{2}$ we have 
$( K^{t} ) = I_{min} (f)$ and $\Lambda^{*} (f)=f^{*}$. 
As $I_{min} (f)$ is zero-dimensional, for $t$ high enough, a border basis $F'$ of $K^{t}$ in degree
$\le 2\, t$ is a border basis of $( K^{t} ) = I_{min} (f)$. Let $B'$
be the corresponding monomial basis. Then $r:=|B'|$ is the number of
minimizers (ie. the points in $V (I_{min} (f))$).
By Lemma \ref{lemextend} with $E=\Span{B'^{+}}$ and Theorem \ref{caszero}, any linear
form $\Lambda'\in \Lc_{F',B,\succcurlyeq}$ generic for $f$ is the
restriction of a positive linear form 
$$ 
\tilde{\Lambda}'= \sum_{i=1}^{r} \lambda_{i}\, \ev_{\zeta_{i}}
$$
with $\{\zeta_{1}, \ldots, \zeta_{r}\} = V (I_{min} (f))$ and
$\lambda_{i}>0$ and $\sum_{i=1}^{r} \lambda_{i}=1$.
In this case, $\Lambda' (f)= \sum_{i=1}^{r} \lambda_{i}\, f
({\zeta_{i}})= f^{*}$ and $\ker H^{B '}_{\Lambda'}=\{0\}$.
We arrive at a contradiction, which shows that the algorithm should stops
in step (4), for some degree $t$. 

By Lemma \ref{algostopbien}, $F'$ is a border basis of $I_{min} (f)$
with basis $B'$ and $\Lambda' (f)= f^{*}$.
\end{proof}

\section{Examples}
This section contains examples which illustrate the behavior of the
algorithm. In the first example of Motzkin polynomial, the gradient ideal is not
zero-dimensional whereas that in the second example of Robinson polynomial, the
gradient ideal is zero-dimensional. In all the examples the minimizer
ideal is zero-dimensional hence when we apply our algorithm we obtain
the good result in a finite number of steps.

The implementation of the previous algorithm has been performed using
the {\sc borderbasix}\footnote{www-sop.inria.fr/galaad/mmx/borderbasix} package of the {\sc Mathemagix}\footnote{www.mathemagix.org} software. {\sc
  borderbasix} is a {\tt C++} implementation of the border basis algorithm of
\cite{BMPhT12}.

Semidefinite positive Hankel operators are computed using the semi
definite programming routine of \textsc{sdpa}\footnote{sdpa.sourceforge.net} software. For the
link with \textsc{sdpa} we use a file interface since \textsc{sdpa} is not
distributed as a library.

For the computation of border basis, we use as a choice function that is tolerant
to numerical instability i.e. a choice function that chooses as
leading monomial a monomial whose coefficient is maximal among the
choosable monomials. This, according to \cite{BMPhT08}, makes the border
basis computation stable with respect to numerical perturbations.
This property is fundamental as we use results from SDP solvers to get
new equations. And as these equations are computed numerically, the
computation is subject to numerical errors. 

Once the border basis of the minimizer is computed, the roots are
obtained using the numerical routines described in \cite{SGPhT09}.

Experiments are made on an Intel Corei7 2.30GHz with 8Gb of RAM.

In the following examples, we use the natation
$\overline{H}_{\Lambda}^t=H_{\Lambda}^{B_t}$

\begin{example}

We consider the Motzkin polynomial
\begin{equation*}
 f(x,y)=1+x^4y^2+x^2y^4-3x^2y^2
\end{equation*}
which is non negative on $\R^2$ but not a sum of squares in $\R[x,y]$.
 We compute its gradient ideal,
 $I_{grad}(f)=(-6xy^2+2xy^4+4x^3y^2,-6yx^2+2yx^4+4y^3x^2)$ 
which is not zero-dimensional.
    
\begin{itemize}
 \item In the first iteration the degree is 3, the size of the Hankel
   matrix $\overline{H}_{\Lambda}^3$ is 10, $\min \Lambda(f) =-216$, there is a duality gap hence we try with degree 4.
 \item In the second iteration the degree is 4, the size of the Hankel
   matrix $\overline{H}_{\Lambda}^4$ is 15, min $\Lambda(f) =0$, there is no
   duality gap. The minimum differs from the previous minimum, so
   a new iteration is needed.
 \item In the third iteration the degree is 5, the size of the Hankel matrix $\overline{H}_{\Lambda}^5$ is 19 and $\min \Lambda(f) =0$.
       The minimums are equal hence we compute the kernel of
       $\overline{H}_{\Lambda}^4$, which is generated by
       5 polynomials. We compute the border basis 
       and obtain the basis $B=\{1,x,y,xy\}$. All the elements of $B$ have degree $<4$. 
 \item We compute a generic form for $f$ with this border basis, the
   size of the Hankel matrix $\overline{H}_{\Lambda}^4$ is 4, $\min \Lambda(f)$ =0 and $\ker \overline{H}_{\Lambda}^4=\{0\}$.
\end{itemize}

 After the fourth iteration the algorithm stops and we obtain 
   \begin{enumerate}
   \item $I_{min}=(x^2-1,y^2-1)$.
    \item The basis $B=\{1,x,y,xy\}$.
    \item The points which minimize f, $\{(x=1,y=1),(x=1,y=-1),(x=-1,y=1),(x=-1,y=-1)\}$.
    \end{enumerate}
The complete process of resolution took $0.286s$ on which
$0.154s$ is spent computing SDP.  

\end{example}    
    
\begin{example}
We consider the Robinson polynomial,
\begin{equation*}
 f(x,y)=1+x^6-x^4-x^2+y^6-y^4-y^2-x^4y^2-x^2y^4+3x^2y^2
\end{equation*}
which is non negative on $\R^2$ but not a sum of squares in $\R[x,y]$.
 We compute its gradient ideal
$I_{grad}(f)=(6x^5-4x^3-2x-4x^3y^2-2xy^4+6xy^2,6y^5-4y^3-2y-4y^3x^2-2yx^4+6yx^2)$,\\
which is not zero-dimensional.

\begin{itemize}
 \item In the first iteration the degree is 3, the size of the Hankel matrix $\overline{H}_{\Lambda}^3$ is 10, min $\Lambda(f) =-0.93$, there is no duality gap hence we compare the minimum with degree 4.
 \item In the second iteration the degree is 4, the size of the Hankel
   matrix $\overline{H}_{\Lambda}^4$ is 15, $\min \Lambda(f) =0$. As the minimums
   are different, another iteration is needed.
 \item In the third iteration the degree is 5, the size of the Hankel
   matrix $\overline{H}_{\Lambda}^5$ is 19 and $\min \Lambda(f)$ =0.
       The minimums are equal, hence we compute the kernel of
       $\overline{H}_{\Lambda}^4$, which is
       generated by 6 polynomials. We compute the border basis 
       and we obtain \\ $B=\{1,x,y,x^2,xy,y^2,x^2y,xy^2,x^2y^2\}$. There exists an element of $B$ with degree $\geq 4$, so we go
       to the next degree.
 \item We compute a generic form $\Lambda$ for $f$ in degree $6$, the size of the Hankel
   matrix $\overline{H}_{\Lambda}^6$ is 22 and $\min \Lambda(f)$ =0.
   The minimums are equal, hence we compute the kernel of
   $\overline{H}_{\Lambda}^5$, which is generated by 11 polynomials. We compute the border basis and obtain
   the basis $B=\{1,x,y,x^2,xy,y^2,x^2y,xy^2\}$. All the elements of $B$ have degree $<5$. 
 \item We compute a generic form for $f$ with this border basis, the
   size of the Hankel matrix $\overline{H}_{\Lambda}^5$ is 8, $\min \Lambda(f)$ =0 and $\ker \overline{H}_{\Lambda}^5=\{0\}$.
\end{itemize}

 After fourth iteration the algorithm stops and we obtain 
   \begin{enumerate}
    \item $I_{min}=(x^3-x,y^3-y,x^2y^2-x^2-y^2+1)$.
    \item The basis $B=\{1,x,y,x^2,xy,y^2,x^2y,xy^2\}$.
    \item The points which minimize f, $\{(x=1,y=1),(x=1,y=-1),(x=-1,y=1), (x=-1,y=-1),
        (x=1,y=0),(x=-1,y=0),(x=0,y=1),(x=0,y=-1)\}$.
    \end{enumerate}
The complete process of resolution took $0.315s$ on which $0.134s$ is spent computing SDP.
 \end{example}
 
 \begin{example}
 We consider the polynomial,
 \begin{equation*}
  f(x,y)=-12x^3+3xy^2+4y^3-16x^2y+48x^2-12y^2
 \end{equation*}

  We compute its gradient ideal, $I_{grad}(f)=(-36x^2+3y^2-32xy+96x,6xy+12y^2-16x^2-24y)$\\
 \begin{itemize}
 \item In the first iteration the degree is 2, the size of the Hankel
   matrix $\overline{H}_{\Lambda}^2$ is 4, min $\Lambda(f) =-18.6$.
   The minimum differs from the previous minimum
   $-\infty$. Hence, a new iteration is needed.
 \item In the second iteration the degree is 3, the size of the Hankel
   matrix $\overline{H}_{\Lambda}^3$ is 4 and $\min \Lambda(f) =-18.6$.
       The minimums are equal hence we compute the kernel of
       $\overline{H}_{\Lambda}^2$, which is generated by 3 polynomials. We compute the border basis 
       and we obtain $B=\{1\}$. All the elements of $B$ have degree $<2$.
 \item We compute a generic form for $f$ with this border basis, the
   size of the Hankel matrix $\overline{H}_{\Lambda}^2$ is 1, $\min \Lambda(f) =-18.6$ and $\ker \overline{H}_{\Lambda}^2=\{0\}$.           
\end{itemize}
After second iteration the algorithm stops and we obtain 
   \begin{enumerate}
    \item $I_{min}=(x+0.43636,y-2.32727)$.
    \item The basis $B=\{1\}$.
    \item The points which minimize f, $\{(x=-0.43636,y=2.32727)\}$.
    \end{enumerate}
    
The complete process of resolution took $0.427s$ on which $0.130s$ were spent computing SDP.   
 \end{example}
 
 \begin{example}
 We consider the Leep and Starr polynomial,
 \begin{equation*}
  f(x,y)=16+x^2y^4+2x^2y^3-4x^3y^3+4xy^2+20x^2y^2+8x^3y^2+6x^4y^2+8xy-16x^2y
 \end{equation*}
that is positive on $\R^2$ but cannot be written as sum of squares in $\R[x,y]$.
 We compute its gradient ideal,\\ 
 $I_{grad}(f)=(2xy^4+4xy^3-12x^2y^3+4y^2+40xy^2+24x^2y^2+24x^3y^2+8y-32xy,4x^2y^3+6x^2y^2-12x^3y^2+8xy+40x^2y+16x^3y+12x^4y+8x-16x^2)$\\
 
 \begin{itemize}
 \item In the first iteration the degree is 3, the size of the Hankel matrix $\overline{H}_{\Lambda}^3$ is 10, min $\Lambda(f) =-5.4$, there is duality gap we try with degree 4 but we change the basis we try again with degree 3.
 \item In the second iteration the degree is 3, the size of the Hankel
   matrix $\overline{H}_{\Lambda}^3$ is 10, min $\Lambda(f) = 0.6$, there is
   no duality gap. As the minimum differs from the minimum in degree
   $2$, a new iteration step is performed.
 \item In the third iteration the degree is 4, the size of the Hankel matrix $\overline{H}_{\Lambda}^4$ is 11 (with the reduction) and min $\Lambda(f) =0.6$.
       The minimums are equal, hence we compute the kernel of
       $\overline{H}_{\Lambda}^3$, which is generated by 9 polynomials. We compute the border basis 
       and obtain the basis $B=\{1\}$. All the elements of $B$ have degree $<3$. 
 \item We compute a generic form for $f$ with this border basis, the
   size of the Hankel matrix $\overline{H}_{\Lambda}^3$ is 1, $\min \Lambda(f) =0.6$ and $\ker \overline{H}_{\Lambda}^3=\{0\}$.      
\end{itemize}
After the third iteration the algorithm stops and we obtain 
   \begin{enumerate}
    \item $I_{min}=(x+3.3884,y-0.14347)$.
    \item The basis $B=\{1\}$.
    \item The points which minimize f, $\{(x=-3.3884,y=0.14347)\}$.
    \end{enumerate}
The complete process of resolution took $0.417s$ on which
$0.130s$ were spent computing SDP.   
  \end{example}

The experiments suggest that due to the small size of the matrices
$\overline{H}_\Lambda^i$ most of the resolution time is spent during a classical
border basis computation, we all the more emphasize this, as we used a
file interface for communicating with {\tt sdpa} which is rather
slow.  

%\bibliographystyle{plain}
%\bibliography{paper}

\begin{thebibliography}{10}

\bibitem{BCR98}
J.~Bochnak, M.~Coste, and M.-F. Roy.
\newblock {\em Real Algebraic Geometry}.
\newblock Springer, 1998.

\bibitem{Boyd-Vandenbergh}
S.~Boyd and L.~Vandenbergh.
\newblock {\em Convex Optimization}.
\newblock Cambridge University Press., New York, 2004.

\bibitem{CLO98}
D.A. Cox, J.B. Little, and D.B. O'Shea.
\newblock {\em Using Algebraic Geometry}.
\newblock Springer, 1998.

\bibitem{CLO97}
D.A. Cox, J.B. Little, and D.B. O'Shea.
\newblock {\em Ideals, Varieties, and Algorithms : An Introduction to
  Computational Algebraic Geometry and Commutative Algebra (Undergraduate Texts
  in Mathematics)}.
\newblock Springer, 2005.

\bibitem{CF96}
R.E. Curto and L.~Fialkow.
\newblock Solution of the truncated complex moment problem for flat data.
\newblock {\em Memoirs of the American Mathematical Society}, 119(568):1--62,
  1996.

\bibitem{em-07-irsea}
M.~Elkadi and B.~Mourrain.
\newblock {\em Introduction {\`a} la r\'esolution des syst{\`e}mes
  d'\'equations alg\'ebriques}, volume~59 of {\em Math\'ematiques et
  Applications}.
\newblock Springer-Verlag, 2007.

\bibitem{SGPhT09}
S.~Graillat and Ph. Tr{\'{e}}buchet.
\newblock {A} new algorithm for computing certified numerical approximations of
  the roots of a zero-dimensional system.
\newblock In {\em ISSAC2009}, pages 167--173, July 2009.

\bibitem{Greuet:2011:DRI:1993886.1993910}
A.~Greuet and M.~Safey El~Din.
\newblock Deciding reachability of the infimum of a multivariate polynomial.
\newblock In {\em Proceedings of the 36th international symposium on Symbolic
  and algebraic computation}, ISSAC '11, pages 131--138, New York, NY, USA,
  2011. ACM.

\bibitem{Guo:2010:GOP:1837934.1837960}
F.~Guo, M.~Safey El~Din, and L.~Zhi.
\newblock Global optimization of polynomials using generalized critical values
  and sums of squares.
\newblock In {\em Proceedings of the 2010 International Symposium on Symbolic
  and Algebraic Computation}, ISSAC '10, pages 107--114, New York, NY, USA,
  2010. ACM.

\bibitem{Hanzon01globalminimization}
B.~Hanzon and D.~Jibetean.
\newblock Global minimization of a multivariate polynomial using matrix
  methods.
\newblock {\em J. Global Optim}, 27:1--23, 2001.

\bibitem{Jibetean:2005:SAG:1093657.1108704}
D.~Jibetean and M.~Laurent.
\newblock Semidefinite approximations for global unconstrained polynomial
  optimization.
\newblock {\em SIAM J. on Optimization}, 16(2):490--514, June 2005.

\bibitem{Kaltofen:2012:ECG:2069778.2070141}
E.~L. Kaltofen, B.~Li, Z.~Yang, and L.~Zhi.
\newblock Exact certification in global polynomial optimization via
  sums-of-squares of rational functions with rational coefficients.
\newblock {\em J. Symb. Comput.}, 47(1):1--15, January 2012.

\bibitem{lasserre:hal-00651759}
J.-B. Lasserre, M.~Laurent, B.~Mourrain, P.~Rostalski, and P.~Tr{\'e}buchet.
\newblock Moment matrices, border bases and real radical computation.
\newblock {\em Journal of Symbolic Computation}, 2012.

\bibitem{Las01}
J.B. Lasserre.
\newblock Global optimization with polynomials and the problem of moments.
\newblock {\em SIAM J. Optim.}, 11:796--817, 2001.

\bibitem{LLR07}
J.B. Lasserre, M.~Laurent, and P.~Rostalski.
\newblock Semidefinite characterization and computation of real radical ideals.
\newblock {\em Foundations of Computational Mathematics}, 8(5):607--647, 2008.

\bibitem{LLR08b}
J.B. Lasserre, M.~Laurent, and P.~Rostalski.
\newblock A unified approach for real and complex zeros of zero-dimensional
  ideals.
\newblock In M.~Putinar and S.~Sullivant, editors, {\em Emerging Applications
  of Algebraic Geometry.}, volume 149, pages 125--156. Springer, 2009.

\bibitem{Lau07}
M.~Laurent.
\newblock Semidefinite representations for finite varieties.
\newblock {\em Math. Progr}, 109:1--26, 2007.

\bibitem{MoLa2008}
M.~Laurent and B.~Mourrain.
\newblock A generalized flat extension theorem for moment matrices.
\newblock {\em Arch. Math. (Basel)}, 93(1):87--98, July 2009.

\bibitem{Marshall03}
M.~Marshall.
\newblock Optimization of polynomial functions.
\newblock {\em Canad. Math. Bull.}, 46:575--587, 2003.

\bibitem{m-99-nf}
B.~Mourrain.
\newblock A new criterion for normal form algorithms.
\newblock In M.~Fossorier, H.~Imai, Shu Lin, and A.~Poli, editors, {\em Proc.
  AAECC}, volume 1719 of {\em LNCS}, pages 430--443. Springer, Berlin, 1999.

\bibitem{Mourrain2005}
B.~Mourrain and P.~Tr{\'e}buchet.
\newblock Generalized normal forms and polynomials system solving.
\newblock In M.~Kauers, editor, {\em {ISSAC}: Proceedings of the {ACM} {SIGSAM}
  International Symposium on Symbolic and Algebraic Computation}, pages
  253--260, 2005.

\bibitem{BMPhT08}
B.~Mourrain and Ph. Tr{\'{e}}buchet.
\newblock {S}table normal forms for polynomial system solving.
\newblock {\em {T}heoretical {C}omputer {S}cience}, 409(2):229--240, 2008.

\bibitem{Nesterov2000}
Y.~Nesterov.
\newblock Squared functional systems and optimization problems.
\newblock In H.~Frenk, K.~Roos, T.~Terlaky, and S.~Zhang, editors, {\em High
  performance optimization}, chapter~17, pages 405--440. Kluwer academic
  publishers, Dordrecht, The Netherlands, 2000.

\bibitem{Nie11}
J.~Nie.
\newblock An exact jacobian sdp relaxation for polynomial optimization.
\newblock {\em Mathematical Programming}, pages 1--31, 2011.

\bibitem{NDS}
J.~Nie, J.~Demmel, and B.~Sturmfels.
\newblock Minimizing polynomials via sum of squares over gradient ideal.
\newblock {\em Math. Program.}, 106(3):587--606, 2006.

\bibitem{Par03}
P.A. Parrilo.
\newblock Semidefinite programming relaxations for semialgebraic problems.
\newblock {\em Mathematical Programming Ser. B}, 96(2):293--320, 2003.

\bibitem{Parrilo03minimizingpolynomial}
P.A. Parrilo and B.~Sturmfels.
\newblock Minimizing polynomial functions.
\newblock In {\em Proceedings of the DIMACS Workshop on Algorithmic and
  Quantitative Aspects of Real Algebraic Geometry in Mathematics and Computer
  Science}, pages 83--100. American Mathematical Society, 2003.

\bibitem{rudin-functional}
W.~Rudin.
\newblock {\em Functional analysis}.
\newblock International Series in Pure and Applied Mathematics. McGraw-Hill
  Inc., New York, second edition, 1991.

\bibitem{ElDin:2008:CGO:1390768.1390781}
M.~Safey El~Din.
\newblock Computing the global optimum of a multivariate polynomial over the
  reals.
\newblock In {\em Proceedings of the twenty-first international symposium on
  Symbolic and algebraic computation}, ISSAC '08, pages 71--78, New York, NY,
  USA, 2008. ACM.

\bibitem{Schweighofer06}
M.~Schweighofer.
\newblock Global optimization of polynomials using gradient tentacles and sums
  of squares.
\newblock {\em SIAM Journal on Optimization}, 17(3):920--942, 2006.

\bibitem{Shor87}
N.Z. Shor.
\newblock Class of global minimum bounds of polynomial functions.
\newblock {\em Cybernetics}, 23:731--734, 1987.

\bibitem{BMPhT12}
Ph. Tr{\'{e}}buchet and B.~Mourrain.
\newblock {B}order basis representation of a general quotient algebra.
\newblock In Joris van~der Hoeven, editor, {\em {ISSAC} 2012}, pages 265--272,
  July 2012.

\end{thebibliography}

\end{document}

% \begin{thebibliography}{999}
% %\bibliography{frooble}
% %\addcontentsline{toc}{chapter}{Bibliography}
% \bibitem{Las} J.B. Lasserre. \textit{Global optimization with polynomials an the problem of moments}.
% \bibitem{Lau} J.B Lasserre, M. Laurent, P. Rostalki. \textit{Semidefinite characterization and computation of zero-dimensional real radical ideals}. 
% \bibitem{Mou} J.B Lasserre, M. Laurent, B. Mourrain, P. Rostalki and P. Trebuchet. \textit{Moment matrices, Border Bases and Real Radical computation}.
% \bibitem{Nie} J. Nie, J. Demmel and B. Sturmfelds. \textit{Minimizing Polynomials via Sum of Squares over the Gradient Ideal} 
% \bibitem{MoLa} M. Laurent and B. Mourrain. \textit{A generalized flat extension theorem for moment matrices}
% \end{thebibliography}